\documentclass[12pt,a4paper,reqno]{amsart} %optio titlepage!

\usepackage[T1]{fontenc}       % kirjaimet, joissa aksentteja (skandit)

\usepackage{amsfonts}          % AMS-fontit

\usepackage{amsmath}

\usepackage{amssymb}

\usepackage{overpic}
\usepackage{euscript} % \mathcal tuottaa hyvдnnдkцisen fontin
\usepackage{eurosym}
\usepackage{comment}
\usepackage{mathrsfs} % calligraphic font via the command \mathscr{ABC}
\usepackage{ifthen}
\usepackage{esint} %k.a.integraali komennolla \fint
\usepackage{color}

\long\def\symbolfootnote[#1]#2{\begingroup%
\def\thefootnote{\fnsymbol{footnote}}\footnote[#1]{#2}\endgroup}

{\qed\vspace{5pt}}

\usepackage{amsthm}

\newtheoremstyle{lause}% name
{5pt}% space above
{5pt}% space below
{\slshape}% body font
{\parindent}% indent amount (empty = no indent)
{\bfseries}% theorem head font
{.}% punctuation after theorem head
{.5em}% space after theorem head
{}% theorem head spec (can be left empty, meaning 'normal')

\theoremstyle{lause}
\newtheoremstyle{maaritelma}% name
{5pt}% space above
{5pt}% space below
{\rmfamily}% body font
{\parindent}% indent amount (empty = no indent)
{\bfseries}% theorem head font
{.}% punctuation after theorem head
{.5em}% space after theorem head
{}% theorem head spec (can be left empty, meaning 'normal')

\theoremstyle{maaritelma}
\newtheoremstyle{lause}% name
{5pt}% space above
{5pt}% space below
{\slshape}% body font
{\parindent}% indent amount (empty = no indent)
{\bfseries}% theorem head font
{.}% punctuation after theorem head
{.5em}% space after theorem head
{}% theorem head spec (can be left empty, meaning 'normal')

\theoremstyle{lause}
\newtheorem{theorem}{Theorem}[section]
\newtheorem{lemma}[theorem]{Lemma}
\newtheorem{proposition}[theorem]{Proposition}
\newtheorem{corollary}[theorem]{Corollary}
\newtheorem{problem}[theorem]{Problem}

\newtheoremstyle{maaritelma}% name
{5pt}% space above
{5pt}% space below
{\rmfamily}% body font
{\parindent}% indent amount (empty = no indent)
{\bfseries}% theorem head font
{.}% punctuation after theorem head
{.5em}% space after theorem head
{}% theorem head spec (can be left empty, meaning 'normal')

\theoremstyle{maaritelma}
\newtheorem{definition}[theorem]{Definition}

\newtheorem{example}[theorem]{Example}
\newtheorem{remark}[theorem]{Remark}

\numberwithin{equation}{section}

\begin{document}

\thispagestyle{empty}

\begin{center}

{\large{\textbf{Minimum energy problems with external fields\\ on locally compact spaces
}}}

\vspace{18pt}

\textbf{Natalia Zorii}

\vspace{18pt}

\emph{In memory of Makoto Ohtsuka (1922--2007)}\vspace{8pt}

\footnotesize{\address{Institute of Mathematics, Academy of Sciences
of Ukraine, Tereshchenkivska~3, 01601,
Kyiv-4, Ukraine\\
natalia.zorii@gmail.com }}

\end{center}

\vspace{12pt}

{\footnotesize{\textbf{Abstract.} The paper deals with minimum energy problems in the presence of external fields on a locally compact space $X$ with respect to a function kernel $\kappa$ satisfying the energy and consistency principles.
For quite a general (not necessarily lower semicontinuous) external field $f$, we establish sufficient and/or necessary conditions for the existence of $\lambda_{A,f}$ minimizing the Gauss functional
\[\int\kappa(x,y)\,d(\mu\otimes\mu)(x,y)+2\int f\,d\mu\]
over all positive Radon measures $\mu$ with $\mu(X)=1$, concentrated on quite a general (not necessarily closed or bounded) $A\subset X$, thereby giving an answer to a question raised by M.~Ohtsuka (J.~Sci.\ Hiroshima Univ., 1961). Such results are specified for the Riesz kernels $|x-y|^{\alpha-n}$, $0<\alpha<n$, on $\mathbb R^n$, $n\geqslant2$, and are illustrated by some examples. Furthermore, we provide various alternative characterizations of the minimizer $\lambda_{A,f}$, and as a by-pro\-duct we analyze the strong and vague continuity of $\lambda_{A,f}$ under the exhaustion of $A$ by compact $K\subset A$.
The results obtained hold true and are new for many interesting kernels in classical and modern potential theory.
}}
\symbolfootnote[0]{\quad 2010 Mathematics Subject Classification: Primary 31C15.}
\symbolfootnote[0]{\quad Key words: Radon measures on locally compact spaces, minimum energy problems, external fields, perfect kernels, maximum principles, inner capacitary measures, inner balayage.
}

\vspace{6pt}

\markboth{\emph{Natalia Zorii}} {\emph{Minimum energy problems with external fields}}

\section{Statement of the problem. Main results}\label{sec-intr}

C.F.~Gauss investigated the variational
problem of minimizing the Newtonian energy evaluated in the presence
of an external field, nowadays called the Gauss functional (or, in constructive function theory, the weighted energy), over
positive charges $\varphi\,ds$ on the boundary surface of a bounded
domain in $\mathbb R^3$ (see \cite{Gau}).

A far-rea\-ch\-ing generalization of the original Gauss
variational problem, employing vec\-tor-val\-ued Radon measures $\boldsymbol{\mu}=(\mu^i)_{i\in I}$ on a locally compact space $X$ as charges and replacing the Newtonian
kernel by a function kernel $\kappa$ on $X$, has grown into an eminent branch of modern potential theory, initiated in the fundamental work by M.~Ohtsuka \cite{O}. See e.g.\ the author's papers \cite{Z5a}--\cite{ZPot3}, where vec\-tor-val\-ued Radon measures $\boldsymbol{\mu}$ were even allowed to be infinite dimensional. Regarding the analytic, constructive, and numerical analysis of the Gauss variational problem for scalar Borel measures on $\mathbb R^n$, $n\geqslant2$, with respect to the logarithmic or Riesz kernels, see the monographs \cite{BHS,ST} and numerous references therein, as well as \cite{Dr0}, \cite{HWZ1}--\cite{HWZ3}.

Throughout the present paper, $X$ denotes a locally compact (Hausdorff) space, $\mathfrak M$ the linear space of all (re\-al-val\-ued scalar Radon) measures $\mu$ on $X$ equipped with the {\it vague\/} ({}=\,{\it weak\/$^*$}) topology of pointwise convergence on the class $C_0(X)$ of all continuous functions $\varphi:X\to\mathbb R$ of compact support, and $\mathfrak M^+$ the cone of all positive $\mu\in\mathfrak M$, where $\mu\in\mathfrak M$ is {\it positive\/} if and only if $\mu(\varphi)\geqslant0$ for all positive $\varphi\in C_0(X)$.

A {\it kernel\/} $\kappa$ on $X$ is meant to be a symmetric function from $\Phi(X\times X)$, where $\Phi(Y)$ consists of all lower semicontinuous (l.s.c.) functions $g:Y\to(-\infty,\infty]$ such that $g\geqslant0$ unless the topological space $Y$ is compact. Then {\it the energy}
\begin{equation*}
 I(\mu):=\int\kappa(x,y)\,d(\mu\otimes\mu)(x,y),\quad\mu\in\mathfrak M,
\end{equation*}
with respect to the kernel $\kappa$ is well defined (as a finite number or $+\infty$) on all of $\mathfrak M^+$, and represents there a vaguely l.s.c.\ function (see Section~\ref{sec-1} below for more details). We denote by $\mathcal E^+$ the set of all $\mu\in\mathfrak M^+$ with $I(\mu)<\infty$.

For any $A\subset X$, denote by $\mathcal E^+(A)$ the class of all $\mu\in\mathcal E^+$ {\it concentrated on\/} $A$ \cite[Section~V.5.7]{B2}, and by $\breve{\mathcal E}^+(A)$ its subclass consisting of all $\mu$ with $\mu(X)=1$.

\subsection{Statement of the problem} Fix a universally measurable function $f:X\to[-\infty,\infty]$, to be treated as an external field acting on charges (measures) on $X$. Given $A\subset X$, let $\breve{\mathcal E}^+_f(A)$ stand for the class of all $\mu\in\breve{\mathcal E}^+(A)$ such that $f$ is $\mu$-in\-te\-g\-r\-ab\-le \cite{B2} (Chapter~IV, Sections~3, 4).
Then {\it the $f$-wei\-gh\-t\-ed energy\/} ({}=\,{\it the Gauss functional\/})
\begin{equation}\label{fen}
 I_f(\mu):=I(\mu)+2\int f\,d\mu
\end{equation}
is well defined and finite for all $\mu\in\breve{\mathcal E}^+_f(A)$, and one can introduce the extremal value\footnote{As usual, the infimum over the empty set is interpreted as $+\infty$. We also agree that $1/(+\infty)=0$ and $1/0 = +\infty$.}
\begin{equation}\label{wf}
w_f(A):=\inf_{\mu\in\breve{\mathcal E}^+_f(A)}\,I_f(\mu)\in[-\infty,\infty].
\end{equation}

In the remainder of the present section, we shall always assume that
\begin{equation}\label{fin}
-\infty<w_f(A)<\infty.
\end{equation}
(See Lemmas~\ref{l-aux-2}, \ref{l-aux-3} and Corollary~\ref{l-aux-4} for necessary and/or sufficient conditions for this to hold.) Then $\breve{\mathcal E}^+_f(A)\ne\varnothing$, and hence the following problem makes sense.

\begin{problem}\label{pr-main} Does there exist\/ $\lambda=\lambda_{A,f}\in\breve{\mathcal E}^+_f(A)$ with
\[I_f(\lambda_{A,f})=w_f(A)?\]
\end{problem}

Problem~\ref{pr-main} is often referred to as {\it the inner Gauss variational problem\/} \cite{O}.
We call its solutions $\lambda_{A,f}$ (if they exist) {\it the inner\/ $f$-weighted equilibrium measures\/} of $A$.

Assume for a moment that $A=K\subset X$ is compact, and that $f\in\Phi(X)$. Then the solutions $\lambda_{K,f}$ do exist,\footnote{In general, such $\lambda_{K,f}$ are {\it not\/} unique (unless of course the kernel $\kappa$ satisfies the energy principle).}
which follows easily from the vague compactness of the class of admissible measures,
and the vague lower semicontinuity of the $f$-we\-i\-g\-h\-t\-ed energy $I_f(\mu)$ on $\mathfrak M^+$.
However, this fails to hold if either $A$ is noncompact or $f\not\in\Phi(X)$, and then Problem~\ref{pr-main} becomes "rather difficult" (Ohtsuka \cite[Section~2.2]{O}).

In the remainder of this subsection as well as throughout Section~\ref{subs-main}, we impose on $\kappa$ the following permanent requirement:
\begin{itemize}
  \item[$(H_1)$] The kernel $\kappa$ is {\it perfect}, or equivalently it satisfies {\it the energy and consistency principles\/} (see \cite{F1}, cf.\ Section~\ref{subs-pr} below).
\end{itemize}
Then all (signed) measures $\mu\in\mathfrak M$ with $I(\mu)<\infty$ form a pre-Hil\-bert space $\mathcal E$ with the inner product $\langle\mu,\nu\rangle:=\int\kappa(x,y)\,d(\mu\otimes\nu)(x,y)$
and the energy norm $\|\mu\|:=\sqrt{I(\mu)}$. Moreover, the cone $\mathcal E^+$ then becomes complete in the strong topology, determined by this norm, and the strong topology on $\mathcal E^+$ is finer than the (induced) vague topology.

These facts made it possible to develop the theory of inner capacitary measures as well as that of inner balayage, the latter however additionally requiring the (perfect) kernel $\kappa$ to satisfy the domination principle (B.~Fuglede \cite{F1} and N.~Zorii \cite{Z-arx1}--\cite{Z-arx}, cf.\ Sections~\ref{sec-cap}, \ref{sec-bal} below). Along with the advantages of using perfect kernels, these two theories are crucial to the analysis of Problem~\ref{pr-main}, performed in the present work.

\subsection{Main results}\label{subs-main} Maintaining the interaction between the strong and the vague topologies on the (strongly complete) cone  $\mathcal E^+$ as the main tool in the present study, we obtain sufficient and/or necessary conditions for the solvability of Problem~\ref{pr-main} for {\it noncompact\/} (and {\it even nonclosed\/}) sets $A\subset X$ and for quite general ({\it not necessarily lower semicontinuous\/}) external fields $f$, thereby giving an answer to the ab\-o\-ve-quo\-t\-ed question by Ohtsuka \cite[Section~2.2]{O} (see Theorems~\ref{th2'}, \ref{th3'}, \ref{riesz} and Corollaries~\ref{qcomp}, \ref{th5'} below).
Furthermore, we establish various alternative characterizations of the solution $\lambda_{A,f}$ to Problem~\ref{pr-main} (see Theorems~\ref{th2'}, \ref{th3'}), and as a by-pro\-d\-uct we prove its strong and vague continuity
under the exhaustion of $A$ by compact $K\subset A$ (Theorem~\ref{th4'}),
thereby justifying the term "{\it inner\/} $f$-wei\-gh\-t\-ed equilibrium measure".

For this purpose, we impose on $f$ the following permanent requirement:
\begin{itemize}
  \item[$(H_2)$] The external field $f$ is representable in the form\footnote{In the recent works on Problem~\ref{pr-main}, mainly dealing with the logarithmic or Riesz kernels on $\mathbb R^n$ (see \cite{BHS,Dr0,ST} and references therein), the external field $f$ is always required to be of the class $\Phi(\mathbb R^n)$, whereas the presence of an alternative/additional source of energy, generated by a signed charge $\vartheta\in\mathcal E$, cf.\ (\ref{fform}) or (\ref{fform'}), agrees well with the original electrostatic nature of the problem.}
\begin{equation}\label{fform}
 f=\psi+U^\vartheta, \text{ \ where $\psi\in\Phi(X)$ and $\vartheta\in\mathcal E$,}
\end{equation}
$U^\vartheta(\cdot):=\int\kappa(\cdot,y)\,d\vartheta(y)$ being {\it the potential\/} with respect to the kernel $\kappa$.
\end{itemize}
(As shown in Lemma~\ref{l-aux-3} below, assumption (\ref{fin}) is then equivalent to $w_f(A)<\infty$, which in turn holds if and only if ${\rm cap}_*\bigl(\{x\in A:\ \psi(x)<\infty\}\bigr)>0$. Here and in the sequel, ${\rm cap}_*(\cdot)$ denotes {\it the inner capacity\/} of a set with respect to the kernel $\kappa$.)

Along with $(H_1)$ and $(H_2)$, some/all of the following three hypotheses on $\kappa$, $A$, and $f$ will often be required:
\begin{itemize}
\item[$(H_3)$] The class $\mathcal E^+(A)$ is closed in the strong topology on $\mathcal E^+$. (This in particular occurs if  $A\subset X$ is {\it quasiclosed\/} ({\it quasicompact\/}), that is, if $A$ can be approximated in {\it outer capacity\/} by closed (compact) sets \cite{F71}; see Section~\ref{def-quasi} below.)\smallskip
\item[$(H_4)$] The (perfect) kernel $\kappa$ satisfies {\it the first and the second maximum principles\/} (for definitions see Section~\ref{subs-pr}).\smallskip
  \item[$(H_2')$] The following particular case of (\ref{fform}) takes place:
\begin{equation}\label{fform'}
f=-U^\zeta,\text{ \ where $\zeta\in\mathcal E^+$ and $\zeta(X)\leqslant1$}.
\end{equation}
\end{itemize}

\begin{theorem}\label{th2'} Assume that\/ $(H_1)$--$(H_3)$ are fulfilled, and that
\begin{equation}\label{cafi'}
{\rm cap}_*(A)<\infty.
\end{equation}
Then the solution\/ $\lambda=\lambda_{A,f}$ to Problem\/~{\rm\ref{pr-main}} does exist.{\rm(\footnote{The solution to Problem~\ref{pr-main} is unique (if it exists), which follows easily from the convexity of the class of admissible measures and the energy principle (see \cite[Lemma~6]{Z5a}, cf.\ Section~\ref{sec-ext} below).})(\footnote{$\lambda_{A,f}$ would not necessarily exist if $(H_3)$ were omitted from the hypotheses (cf.\ footnote~\ref{in-f}, pertaining to the unweighted case $f=0$ and the Newtonian kernel $|x-y|^{2-n}$ on $\mathbb R^n$, $n\geqslant3$).})} Furthermore, $\lambda_{A,f}$ is uniquely determined within\/ $\breve{\mathcal E}^+_f(A)$ by either of the two characteristic properties\/\footnote{For more details about such characteristic inequalities for $\lambda_{A,f}$, see the author's earlier paper \cite{Z5a} (cf.\  Theorems~\ref{th-ch1} and \ref{th-ch2} below). As seen from there, the latter part of Theorem~\ref{th2'} actually holds true under much more general assumptions than stated.}
\begin{align*}
  U_f^{\lambda}&\geqslant c_{A,f}\text{ \ n.e.\ on\/ $A$},\\
  U_f^{\lambda}&\leqslant c_{A,f}\text{ \ $\lambda$-a.e.\ on\/ $X$},
\end{align*}
where\/ $U_f^{\lambda}:=U^\lambda+f$ denotes the\/ $f$-weighted potential of\/ $\lambda$, and
\begin{equation}\label{cc'}
c_{A,f}:=\int U_f^{\lambda}\,d\lambda=w_f(A)-\int f\,d\lambda\in(-\infty,\infty)
\end{equation}
is said to be the inner\/ $f$-weighted equilibrium constant for the set\/ $A$.
\end{theorem}

Here and in the sequel, the abbreviation {\it n.e.}\ ({\it nearly everywhere\/}) means, as usual, that the set of all $x\in A$ where the inequality fails is of inner capacity zero.

\begin{corollary}\label{qcomp}
 Under\/ {\rm(}permanent\/{\rm)} requirements\/ $(H_1)$ and\/ $(H_2)$, Problem\/~{\rm\ref{pr-main}} is\/ {\rm(}uniquely\/{\rm)} solvable for any quasicompact\/ $A\subset X$.
\end{corollary}

\begin{proof}
For quasicompact $A$, (\ref{cafi'}) necessarily holds, the capacity of a compact set with respect to a strictly positive definite kernel being finite. Since $(H_3)$ is fulfilled as well (see Theorem~\ref{l-quasi} below), the corollary follows directly from Theorem~\ref{th2'}.
\end{proof}

Given $A\subset X$, denote by $\mathfrak C_A$ the upward directed set of all compact subsets $K$ of $A$, where $K_1\leqslant K_2$ if and only if $K_1\subset K_2$. If a net $(x_K)_{K\in\mathfrak C_A}\subset Y$ converges to $x_0\in Y$, $Y$ being a topological space, then we shall indicate this fact by writing
\begin{equation*}x_K\to x_0\text{ \ in $Y$ as $K\uparrow A$}.\end{equation*}

\begin{theorem}\label{th4'} Under\/ {\rm(}permanent\/{\rm)} requirements\/ $(H_1)$ and\/ $(H_2)$, assume that the solution\/ $\lambda_{A,f}$ to Problem\/~{\rm\ref{pr-main}} exists.\footnote{See Theorems~\ref{th2'}, \ref{th3'}, \ref{riesz} and Corollaries~\ref{qcomp}, \ref{th5'} for sufficient conditions for this to occur.} Then
\begin{equation}\label{conv2}\lambda_{K,f}\to\lambda_{A,f}\text{ \ strongly and vaguely in\/ $\mathcal E^+$ as\/ $K\uparrow A$},\end{equation}
$\lambda_{K,f}$ being the solution to Problem\/~{\rm\ref{pr-main}} with\/ $A:=K$ {\rm(}such solutions\/ $\lambda_{K,f}$ do exist for all\/ $K\in\mathfrak C_A$ large enough\/{\rm)}. If moreover\/ $f=U^\vartheta$ with\/ $\vartheta\in\mathcal E$, then also
\begin{equation}\label{conv3}\lim_{K\uparrow A}\,c_{K,f}=c_{A,f},\end{equation}
$c_{A,f}$ and\/ $c_{K,f}$ being introduced by formula\/ {\rm(\ref{cc'})} applied to\/ $A$ and\/ $K$, respectively.
\end{theorem}

\begin{theorem}\label{th3'}Let\/ $(H_1)$, $(H_3)$, $(H_4)$, and\/ $(H_2')$ be fulfilled, and let\/ $\zeta^A$ denote the inner balayage of\/ $\zeta$ onto\/ $A$, the measure\/ $\zeta$ appearing in\/ $(H_2')$. If moreover\/\footnote{Under hypotheses $(H_1)$, $(H_3)$, and $(H_4)$, assumption $\zeta^A(X)=1$ is fulfilled, for instance, if $\zeta\in\mathcal E^+$ is a measure of unit total mass {\it concentrated on\/} $A$, i.e.\ $\zeta\in\breve{\mathcal E}^+(A)$. Yet another possibility, pertaining to the $\alpha$-Riesz kernels $|x-y|^{\alpha-n}$ on $\mathbb R^n$, $n\geqslant2$, where $0<\alpha<n$ and $\alpha\leqslant2$, requires $\zeta\in\mathcal E^+$ to be an {\it arbitrary\/} measure on $\mathbb R^n$ with $\zeta(\mathbb R^n)=1$, and $A$ to be {\it not inner\/ $\alpha$-thin at infinity}, which according to \cite{KM,Z-bal2} means that
\[\sum_{k\in\mathbb N}\,\frac{{\rm cap}_*(A_k)}{q^{k(n-\alpha)}}=\infty,\]
where $q\in(1,\infty)$ and $A_k:=A\cap\{x\in\mathbb R^n:\ q^k\leqslant|x|<q^{k+1}\}$. For the latter, see \cite[Corollary~5.3]{Z-bal2}.}
\begin{equation}\label{bal1}
\text{either \ ${\rm cap}_*(A)<\infty$, \ or \ $\zeta^A(X)=1$,}
\end{equation}
then, and only then, the solution\/ $\lambda_{A,f}$ to Problem\/~{\rm\ref{pr-main}} does exist, and it has the form\/\footnote{Representation (\ref{RR}) is particularly useful in applications. For instance, for the $\alpha$-Riesz kernel $|x-y|^{\alpha-n}$ on $\mathbb R^n$, $n\geqslant2$, of order $0<\alpha\leqslant2$, $\alpha<n$, it enables us to give an answer to Open question~2.1 raised by Ohtsuka in
\cite[Section~2.12]{O}. In view of \cite[Theorem~8.5]{Z-bal} and \cite[Theorem~4.2, Corollary~5.4]{Z-bal2}, such answer is in the positive if $\alpha<2$, and it is in the negative otherwise.}
\begin{equation}\label{RR}\lambda_{A,f}=\left\{
\begin{array}{cl} \zeta^A+\eta_{A,f}\gamma_A&\text{if \ ${\rm cap}_*(A)<\infty$},\\
\zeta^A&\text{otherwise},\\ \end{array} \right.
\end{equation}
where\/ $\gamma_A$ denotes the inner capacitary measure on\/ $A$, while
\begin{equation}\label{eqalt}
\eta_{A,f}:=\frac{1-\zeta^A(X)}{{\rm cap}_*(A)}\in[0,\infty).
\end{equation}

Furthermore, the above\/ $\lambda_{A,f}$ can alternatively be characterized by means of any one of the following three assertions:
\begin{itemize}
\item[{\rm(i)}] $\lambda_{A,f}$ is the unique measure of minimum energy norm in the class
\begin{equation}\label{gamma}
\Lambda_{A,f}:=\bigl\{\mu\in\mathcal E^+: \ U^\mu_f\geqslant\eta_{A,f}\text{ \ n.e.\ on\ $A$}\bigr\},
  \end{equation}
$\eta_{A,f}$ being introduced by formula\/ {\rm(\ref{eqalt})}. That is, $\lambda_{A,f}\in\Lambda_{A,f}$ and
\begin{equation*}
\|\lambda_{A,f}\|=\min_{\mu\in\Lambda_{A,f}}\,\|\mu\|.
\end{equation*}
\item[{\rm(ii)}] $\lambda_{A,f}$ is the unique measure of minimum potential in the class\/ $\Lambda_{A,f}$, introduced by means of\/ {\rm(\ref{gamma})}. That is, $\lambda_{A,f}\in\Lambda_{A,f}$ and\/\footnote{This implies immediately that
$\lambda_{A,f}$ can also be characterized as the unique measure of minimum $f$-we\-i\-g\-h\-t\-ed potential in the class $\Lambda_{A,f}$ --- now, however, nearly everywhere on $X$: \begin{equation*} U^{\lambda_{A,f}}_f=\min_{\mu\in\Lambda_{A,f}}\,U^\mu_f\text{ \ n.e.\ on\ $X$}\vspace{-5mm}.\end{equation*}}
\begin{equation*}
U^{\lambda_{A,f}}=\min_{\mu\in\Lambda_{A,f}}\,U^\mu\text{ \ on\ $X$}.
\end{equation*}
\item[{\rm(iii)}] $\lambda_{A,f}$ is the only measure in\/ $\mathcal E^+(A)$ having the property
\begin{equation*}
U^{\lambda_{A,f}}_f=\eta_{A,f}\text{ \ n.e.\ on\ $A$}.
\end{equation*}
\end{itemize}

In addition, the inner\/ $f$-weighted equilibrium constant\/ $c_{A,f}$, introduced by\/ {\rm(\ref{cc'})}, admits an alternative representation
\begin{equation}\label{const-alt}c_{A,f}=\eta_{A,f},\end{equation}
and hence\/ {\rm(\ref{conv3})} can be specified as follows:
\begin{equation}\label{conv3'}c_{K,f}\downarrow c_{A,f}\text{ \ in $\mathbb R$ as $K\uparrow A$}.\end{equation}
\end{theorem}

\begin{corollary}\label{cortm}
Under the assumptions of Theorem\/~{\rm\ref{th3'}}, if moreover\/ $X$ is\/ $\sigma$-co\-m\-p\-act,\footnote{A locally compact space is said to be {\it $\sigma$-com\-pact\/} if it is representable as a countable union of compact sets \cite[Section~I.9, Definition~5]{B1}.}
 then\/ $\lambda_{A,f}$ is a measure of minimum total mass in the class\/ $\Lambda_{A,f}$, i.e.
\begin{equation}\label{eq-t-m}\lambda_{A,f}(X)=\min_{\mu\in\Lambda_{A,f}}\,\mu(X)\quad\bigl({}=1\bigr).\end{equation}
\end{corollary}

\begin{remark}\label{t-m-nonun} However, extremal property (\ref{eq-t-m}) cannot serve as an alternative characterization of the inner $f$-weighted equilibrium measure $\lambda_{A,f}$, for it does not determine $\lambda_{A,f}$ uniquely within $\Lambda_{A,f}$. Indeed, consider the $\alpha$-Riesz kernel $|x-y|^{\alpha-n}$ of order $\alpha\leqslant2$, $\alpha<n$, on $\mathbb R^n$, $n\geqslant2$, a proper, closed subset $A$ of $\mathbb R^n$ that is not $\alpha$-thin at infinity (take, for instance, $A:=\{|x|\geqslant 1\}$), and let $f$ be given by (\ref{fform'}) with $\zeta\in\breve{\mathcal E}^+(\mathbb R^n\setminus A)$. Applying \cite[Corollary~5.3]{Z-bal2} we get
\begin{equation}\label{eq-t-m1}\zeta^A(\mathbb R^n)=\zeta(\mathbb R^n)=1.\end{equation}
Noting that then $\eta_{A,f}=0$, and hence $\zeta,\zeta^A\in\Lambda_{A,f}$ (cf.\ (\ref{eq-pr-10})), we conclude from (\ref{eq-t-m1}) that there are actually infinitely many measures of minimum total mass in $\Lambda_{A,f}$, for so is every measure of the form $a\zeta+b\zeta^A$, where $a,b\in[0,1]$ and $a+b=1$.\end{remark}

Let hypotheses $(H_1)$, $(H_3)$, $(H_4)$, and $(H_2')$ be fulfilled; and let ${\rm cap}_*(A)=\infty$, for if not, the solution $\lambda_{A,f}$ to Problem~\ref{pr-main} does exist by Theorem~\ref{th2'} or \ref{th3'}. Then the following corollary to Theorem~\ref{th3'} holds, where $\zeta$ is the measure appearing in (\ref{fform'}).

\begin{corollary}\label{th5'} The minimizer\/ $\lambda_{A,f}$ does not exist if\/ $\zeta(X)<1$, and it exists if\/ $\zeta$ is a measure of unit total mass concentrated on\/ $A$.\footnote{The latter assertion would fail in general if $\zeta$ were not required to be concentrated on $A$. See Theorem~\ref{riesz}(e); for illustration, see the last claim in Example~\ref{ex} (pertaining to the set $F_2$).} In the latter case, $\lambda_{A,f}=\zeta$.\end{corollary}

\begin{remark}\label{prr}All the above-quoted results hold true and are new for the Green kernels associated with the Laplacian on Greenian sets in $\mathbb R^n$, $n\geqslant2$ (thus in particular for the Newtonian kernel $|x-y|^{2-n}$ on $\mathbb R^n$, $n\geqslant3$), as well as for the $\alpha$-Riesz kernels $|x-y|^{\alpha-n}$ and the associated $\alpha$-Green kernels on $\mathbb R^n$, where $0<\alpha<2\leqslant n$. Furthermore, Theorems~\ref{th2'}, \ref{th4'} and Corollary~\ref{qcomp} also hold true           and are new for the $\alpha$-Riesz kernels of order $2<\alpha<n$ as well as for the logarithmic kernel $-\log\,|x-y|$ on a closed disc in $\mathbb R^2$ of radius ${}<1$, and for the Deny kernels on $\mathbb R^n$, defined with the aid of Fourier transformation (see condition $(A)$ in \cite[Section~1]{De2}, cf.\ \cite[Section~VI.1.2]{L}).
For more details, see Example~\ref{rem:clas} below.
\end{remark}

\subsection{Applications to the Riesz kernels} Let\/ $X:=\mathbb R^n$, $n\geqslant2$, and let\/ $\kappa(x,y)$ be the $\alpha$-Riesz kernel $|x-y|^{\alpha-n}$ of order $0<\alpha<n$. Then the above-quoted results on the (un)solvability of Problem~\ref{pr-main} can be specified as follows.

\begin{theorem}\label{riesz} Assume that an external field\/ $f$ is of form\/ {\rm(\ref{fform})}, that\/ $A\subset\mathbb R^n$ is quasiclosed\/ {\rm(}or, more generally, that\/ $\mathcal E^+(A)$ is strongly closed\/{\rm)}, and that\/ $w_f(A)<\infty$.
\begin{itemize}
  \item[{\rm(a)}] If moreover\/ ${\rm cap}_*(A)<\infty$, then the solution\/ $\lambda_{A,f}$ to Problem\/~{\rm\ref{pr-main}} does exist.
\end{itemize}

In the remaining case\/ ${\rm cap}_*(A)=\infty$, assume\/ $\alpha\leqslant2$, and consider\/ $f$ of form\/ {\rm(\ref{fform'})}. Then the following\/ {\rm(b)}--{\rm(e)} hold true, where\/ $\zeta$ is the measure appearing in\/ {\rm(\ref{fform'})}.
\begin{itemize}
\item[{\rm(b)}] $\lambda_{A,f}$ exists if and only if\/ $\zeta^A(\mathbb R^n)=1$, and in the affirmative case\/ $\lambda_{A,f}=\zeta^A$.
  \item[{\rm(c)}] $\lambda_{A,f}$ does not exist if\/ $\zeta(\mathbb R^n)<1$, and it exists if\/ $\zeta$ is a measure of unit total mass concentrated on\/ $A$. In the latter case, $\lambda_{A,f}=\zeta$.
  \item[{\rm(d)}] If\/ $A$ is not inner\/ $\alpha$-thin at infinity, then\/ $\lambda_{A,f}$ exists if and only if\/ $\zeta(\mathbb R^n)=1$, and in the affirmative case\/ $\lambda_{A,f}=\zeta^A$.\footnote{If\/ $A$ is not inner\/ $\alpha$-thin at infinity, then the requirement ${\rm cap}_*(A)=\infty$ does hold automatically, cf.\ \cite[Section~2]{Z-bal2}, and hence can be dropped.\label{f-thin}}
      \item[{\rm(e)}] Assume that\/ $\overline{A}:={\rm Cl}_{\mathbb R^n}A$ is\/ $\alpha$-thin at infinity, and that\/ $D:=\mathbb R^n\setminus\overline{A}$ is connected unless\/ $\alpha<2$. Then\/ $\lambda_{A,f}$ does not exist whenever\/ $\zeta(D)>0$.\footnote{Actually, we require $D$ to be connected (when $\alpha=2$) only in order to simplify the formulation. Compare with \cite{Z-bal} (Lemma~7.1 and Theorems~7.2, 8.7).}
\end{itemize}
\end{theorem}

The above results on the solvability of Problem~\ref{pr-main} for the $\alpha$-Riesz kernels, included in the present work rather for illustration purposes, can actually be much strengthened, which we plan to pursue in a subsequent paper.

\begin{example}\label{ex} On $\mathbb R^3$, consider the kernel $1/|x-y|$ and the rotation bodies
\begin{equation*}F_i:=\bigl\{x\in\mathbb R^3: \ 0\leqslant x_1<\infty, \
x_2^2+x_3^2\leqslant\varrho_i^2(x_1)\bigr\}, \ i=1,2,3,\end{equation*}
where
\begin{align*}
\varrho_1(x_1)&:=x_1^{-s}\text{ \ with\ }s\in[0,\infty),\\
\varrho_2(x_1)&:=\exp(-x_1^s)\text{ \ with\ }s\in(0,1],\\
\varrho_3(x_1)&:=\exp(-x_1^s)\text{ \ with\ }s\in(1,\infty).
\end{align*}
As can be derived from estimates in \cite[Section~V.1, Example]{L}, $F_1$ is not $2$-thin at infinity, $F_2$ is $2$-thin at infinity, though has infinite Newtonian capacity, whereas $F_3$ is of finite Newtonian capacity. Therefore, by Theorem~\ref{riesz}, $\lambda_{F_3,f}$ exists for any $f$ of form (\ref{fform}), see (a). Let now $f$ be of form (\ref{fform'}) with the measure $\zeta$ involved. Then $\lambda_{F_1,f}$ exists if and only if $\zeta(\mathbb R^n)=1$, see (d); whereas $\lambda_{F_2,f}$ exists if and only if both $\zeta(\mathbb R^n)=1$ and $S(\zeta)\subset F_2$ hold true, see (c) and (e). (Here $S(\zeta)$ is the support of $\zeta$.)\end{example}

\begin{figure}[htbp]
\begin{center}
\vspace{-.8in}
\hspace{-.1in}\includegraphics[width=4.6in]{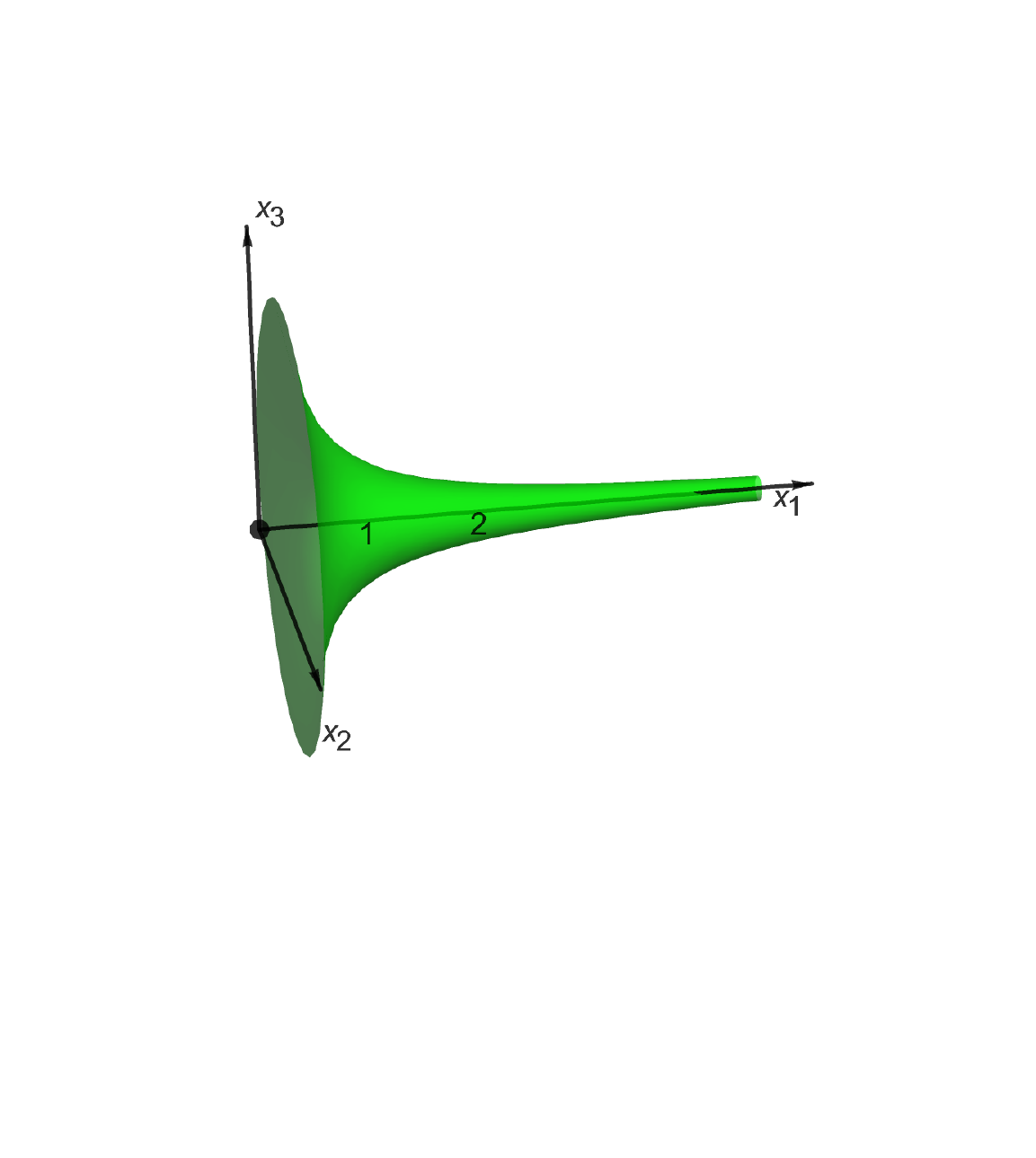}
\vspace{-1.6in}
\caption{The set $F_1$ in Example~\ref{ex} with $\varrho_1(x_1)=1/x_1$.\vspace{-.1in}}
\label{Fig2}
\end{center}
\end{figure}

\begin{figure}[htbp]
\begin{center}
\vspace{-.4in}
\hspace{-1.1in}\includegraphics[width=4.2in]{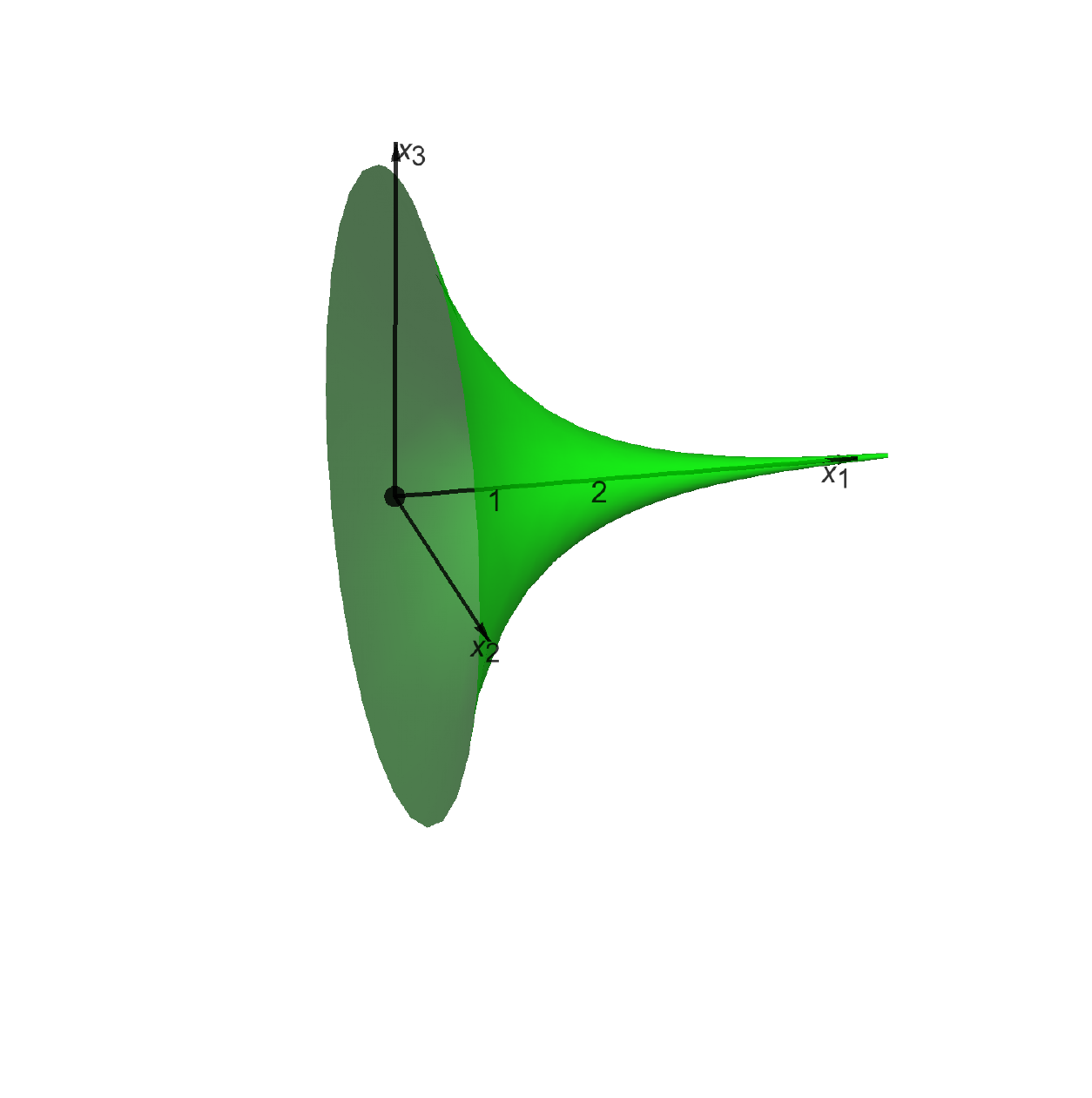}
\vspace{-.8in}
\caption{The set $F_2$ in Example~\ref{ex} with $\varrho_2(x_1)=\exp(-x_1)$.\vspace{-.1in}}
\label{Fig1}
\end{center}
\end{figure}

The rest of the paper is organized as follows. For the convenience of the reader, in Section~\ref{sec-1} we review some basic facts of the theory of potentials on locally compact spaces, the main emphasis being placed on the theory of inner capacitary measures and that of inner balayage. Section~\ref{sec-ext} provides preliminary results on Problem~\ref{pr-main}, and Section~\ref{sec-proofs} gives proofs to Theorems~\ref{th2'}, \ref{th4'}, \ref{th3'}, \ref{riesz} and Corollaries~\ref{cortm}, \ref{th5'}.

\section{On the theory of potentials on locally compact spaces}\label{sec-1}

Throughout the rest of the paper, we shall use the notations and conventions introduced in Section~\ref{sec-intr}.

For the theory of measures and integration on a locally compact (Hausdorff) space $X$, we refer to N.~Bourbaki \cite{B2}.
 A comprehensive application of this theory to the theory of potentials on $X$ with respect to function kernels satisfying the principle of consistency, is presented in the pioneering paper by Fuglede \cite{F1}.

\begin{lemma}[{\rm\cite[Section~IV.1, Proposition~4]{B2}}]\label{lemma-semi}For any l.s.c.\  function\/ $g:X\to[0,\infty]$, the mapping\/ $\mu\mapsto\int g\,d\mu$ is
vaguely l.s.c.\ on\/ $\mathfrak M^+$, the integral here being understood as an upper integral.\end{lemma}

For an arbitrary set $A\subset X$, denote by $\mathfrak M^+(A)$ the cone of all $\mu\in\mathfrak M^+$ {\it concentrated on\/}
$A$, which means that $A^c:=X\setminus A$ is locally $\mu$-neg\-lig\-ible, or equivalently that $A$ is $\mu$-meas\-ur\-able and $\mu=\mu|_A$, $\mu|_A:=1_A\cdot\mu$ being the {\it trace\/} of $\mu$ to $A$ \cite[Section~IV.14.7]{E2}. (Note that for $\mu\in\mathfrak M^+(A)$, the indicator function $1_A$ of $A$ is locally $\mu$-int\-egr\-able.) The total mass of $\mu\in\mathfrak M^+(A)$ is $\mu(X)=\mu_*(A)$, $\mu_*(A)$ and $\mu^*(A)$ denoting the {\it inner\/} and the {\it outer\/} measure of $A$, respectively. If moreover $A$ is closed, or if $A^c$ is contained in a countable union of sets $Q_j$ with $\mu^*(Q_j)<\infty$,\footnote{If the latter holds, $A^c$ is said to be $\mu$-$\sigma$-{\it finite\/} \cite[Section~IV.7.3]{E2}. This in particular occurs if the measure $\mu$ is {\it bounded\/} (that is, with $\mu(X)<\infty$), or if the locally compact space $X$ is $\sigma$-com\-pact.} then for any $\mu\in\mathfrak M^+(A)$, $A^c$ is $\mu$-neg\-lig\-ible, that is, $\mu^*(A^c)=0$. In particular, if $A$ is closed, $\mathfrak M^+(A)$ consists of all $\mu\in\mathfrak M^+$ with support $S(\mu)\subset A$, cf.\ \cite[Section~IV.2.2]{B2}.

\begin{lemma}[{\rm cf.\ \cite[Lemma~1.2.2]{F1}}]\label{l-lower}For any l.s.c.\ function\/ $g:X\to[0,\infty]$, any measure\/ $\mu\in\mathfrak M^+$, and any\/ $\mu$-measurable set\/ $A\subset X$,
\[\int g\,d\mu|_A=\lim_{K\uparrow A}\,\int g\,d\mu|_K.\]
\end{lemma}

A kernel on $X$ is thought of as a symmetric function $\kappa\in\Phi(X\times X)$ (Section~\ref{sec-intr}); thus, either $\kappa(x,y)$ is ${}\geqslant0$ for all $(x,y)\in X\times X$, or the space $X$ is compact.

\begin{remark}\label{fIII}
The case of compact $X$ can mostly be reduced to that of $\kappa\geqslant0$ by replacing the kernel $\kappa$ by $\kappa':=\kappa+q\geqslant0$, where $q\in[0,\infty)$, which is always possible since a lower semicontinuous function on a compact space is lower bounded.\end{remark}

\begin{remark}\label{fIII'}For similar reasons, Lemmas~\ref{lemma-semi} and \ref{l-lower} actually hold true for {\it any\/} $g\in\Phi(X)$ --- even if $g\ngeqslant0$. Indeed, then $X$ must be compact, and hence $g$ can be replaced by $g':=g+q\geqslant0$, where $q\in[0,\infty)$. For Lemma~\ref{lemma-semi}, use the vague continuity of the mapping $\mu\mapsto\mu(X)$ on $\mathfrak M^+$, the space $X$ being compact, while for Lemma~\ref{l-lower}, use the fact that for any $\mu\in\mathfrak M^+$ and any $\mu$-measurable set $A\subset X$,
\[\lim_{K\uparrow A}\,\mu|_K(X)=\mu|_A(X)<\infty.\]
\end{remark}

Given (signed) $\mu,\nu\in\mathfrak M$, define the {\it potential\/} and the {\it mutual energy\/}  by
\begin{gather*}U^\mu(x):=\int\kappa(x,y)\,d\mu(y),\quad x\in X,\\
I(\mu,\nu):=\int\kappa(x,y)\,d(\mu\otimes\nu)(x,y),
\end{gather*}
respectively, provided the right-hand side is well defined as a finite number or $\pm\infty$ (for more details see \cite[Section~2.1]{F1}). For $\mu=\nu$, $I(\mu,\nu)$ defines the {\it energy\/}
$I(\mu):=I(\mu,\mu)$ of $\mu\in\mathfrak M$.
In particular, if $\mu,\nu\geqslant0$, then $U^\mu(x)$, resp.\ $I(\mu,\nu)$, is always well defined and represents a l.s.c.\ function of $(x,\mu)\in X\times\mathfrak M^+$, resp.\ of $(\mu,\nu)\in\mathfrak M^+\times\mathfrak M^+$ (the {\it principle of descent\/} \cite[Lemma~2.2.1]{F1}, cf.\ Lemma~\ref{lemma-semi} and Remark~\ref{fIII'}). Thus
\begin{equation}\label{phi}
 U^\mu\in\Phi(X)\text{ \ for all $\mu\in\mathfrak M^+$}.
\end{equation}

\begin{lemma}\label{222} For any\/ $A\subset X$ and any\/ $\mu\in\mathfrak M^+(A)$,
\begin{equation*}I(\mu)=\lim_{K\uparrow A}\,I(\mu|_K).\end{equation*}
\end{lemma}

\begin{proof}This is seen from Lemma~\ref{l-lower} (cf.\ Remark~\ref{fIII'}) in view of the fact that for any $\mu$-measurable set $Q\subset X$, $(\mu\otimes\mu)|_{Q\times Q}=\mu|_Q\otimes\mu|_Q$, cf.\ \cite[Lemma~1.2.5]{F1}.
\end{proof}

In what follows, we always assume a kernel $\kappa$ to satisfy {\it the energy principle\/} (or equivalently, to be {\it strictly positive definite\/}), which means that $I(\mu)\geqslant0$ for all (signed) $\mu\in\mathfrak M$ whenever $I(\mu)$ is well defined, and moreover $I(\mu)=0$ only for zero measure. Then all $\mu\in\mathfrak M$ with $I(\mu)<\infty$ form a pre-Hil\-bert space $\mathcal E$ with the inner product $\langle\mu,\nu\rangle:=I(\mu,\nu)$ and the energy norm $\|\mu\|:=\sqrt{I(\mu)}$ \cite[Section~3.1]{F1}. The topology on $\mathcal E$ determined by this norm is said to be {\it strong}.

\subsection{Capacities of a set}\label{sec-c} For any $A\subset X$, denote
\begin{gather}\breve{\mathfrak M}^+(A):=\bigl\{\mu\in\mathfrak M^+(A): \ \mu(X)=1\bigr\},\notag\\
\mathcal E^+(A):=\mathcal E\cap\mathfrak M^+(A),\quad
\breve{\mathcal E}^+(A):=\mathcal E\cap\breve{\mathfrak M}^+(A),\notag\\
w(A):=\inf_{\mu\in\breve{\mathcal E}^+(A)}\,I(\mu)\in[0,\infty].\label{W}\end{gather}
The value \[{\rm cap}_*(A):=1/w(A)\in[0,\infty]\] is said to be the (Wiener) {\it inner capacity\/} of the set $A$ (with respect to the kernel $\kappa$).

It is seen from Lemma~\ref{222} that ${\rm cap}_*(A)$ would be the same if the admissible measures in (\ref{W})
were required to be of compact support. This in turn yields\footnote{We write ${\rm cap}(A)$ in place of ${\rm cap}_*(A)$ if $A$ is {\it capacitable}, that is, if ${\rm cap}_*(A)={\rm cap}^*(A)$, where ${\rm cap}^*(A)$ is the {\it outer\/} capacity of $A$, defined as $\inf\,{\rm cap}_*(D)$, $D$ ranging over all open sets containing $A$. This occurs, for instance, if $A$ is compact \cite[Lemma~2.3.4]{F1} or open.}
\begin{equation}\label{153}{\rm cap}(K)\uparrow{\rm cap}_*(A)\text{ \ as $K\uparrow A$}.\end{equation}

\begin{lemma}[{\rm cf.\ \cite[Lemma~2.3.1]{F1}}]\label{l-negl1} For any\/ $A\subset X$, \[{\rm cap}_*(A)=0\iff\mathcal E^+(A)=\{0\}\iff\mathcal E^+(K)=\{0\}\text{ \ for all\/ $K\in\mathfrak C_A$}.\]
\end{lemma}

\begin{lemma}\label{l-negl}Given\/ $\mu\in\mathcal E^+$, let\/ $A\subset X$ be a\/ $\mu$-mea\-sur\-ab\-le and\/ $\mu$-$\sigma$-fi\-ni\-te set with\/ ${\rm cap}_*(A)=0$. Then\/ $A$ is\/ $\mu$-neg\-lig\-ible.\end{lemma}

\begin{proof} As $A$ is $\mu$-mea\-sur\-ab\-le and $\mu$-$\sigma$-fi\-ni\-te, $\mu^*(A)=0$  will follow once we show that $\mu(K)=0$ for every compact $K\subset A$, which however is obvious by Lemma~\ref{l-negl1}.
\end{proof}

A proposition $\mathcal P$ involving a variable point $x\in X$ is said to hold {\it nearly everywhere\/} ({\it n.e.\/}) on $A\subset X$ if ${\rm cap}_*(E)=0$, $E$ being the set of all $x\in A$ where $\mathcal P(x)$ fails. Replacing here ${\rm cap}_*(E)$ by ${\rm cap}^*(E)$ leads to the concept of {\it qua\-si-ev\-ery\-whe\-re\/} ({\it q.e.\/}).

The potential $U^\mu$ of any $\mu\in\mathcal E$ is (well defined and) finite qua\-si-ev\-ery\-whe\-re on $X$ \cite[Corollary to Lemma~3.2.3]{F1}. Furthermore, for any two given $\mu,\nu\in\mathcal E$,
\begin{equation}\label{E}
\mu=\nu\iff U^\mu=U^\nu\text{ \ n.e.\ on $X$}\iff U^\mu=U^\nu\text{ \ q.e.\ on $X$},
\end{equation}
which follows from \cite[Lemma~3.2.1]{F1} by making use of the energy principle.

In the study of inner potential theoretical concepts, the following strengthened version of countable subadditivity for inner capacity is particularly useful.

\begin{lemma}\label{str}
For arbitrary\/ $A\subset X$ and universally measurable\/ $U_j\subset X$, $j\in\mathbb N$,
\[{\rm cap}_*\Bigl(\bigcup_{j\in\mathbb N}\,A\cap U_j\Bigr)\leqslant\sum_{j\in\mathbb N}\,{\rm cap}_*(A\cap U_j).\]\end{lemma}

\begin{proof}Noting that a strictly positive definite kernel is pseudo-positive, cf.\ \cite[p.~150]{F1}, we derive the lemma from \cite{F1} (see Lemma~2.3.5 and the remark following it). For the Newtonian kernel $|x-y|^{2-n}$ on $\mathbb R^n$, $n\geqslant3$, this goes back to H.~Cartan \cite[p.~253]{Ca2}.\end{proof}

\subsection{Potential-theoretic principles}\label{subs-pr} The following two maximum principles are often used throughout the paper.
A kernel $\kappa$ is said to satisfy {\it Frostman's maximum principle\/} ({}=\,{\it the first maximum principle\/}) if for any $\mu\in\mathcal E^+$ with $U^\mu\leqslant1$ on $S(\mu)$, the same inequality holds everywhere on $X$;
and it is said to satisfy {\it the domination principle\/} ({}=\,{\it the second maximum principle\/}) if for any $\mu,\nu\in\mathcal E^+$ with $U^\mu\leqslant U^\nu$ $\mu$-a.e., the same inequality is fulfilled on all of $X$.

For classical kernels, {\it the principle of positivity of mass\/} goes back to J.~Deny \cite{D2}. In the generality stated below, it was established in \cite{Z-arx-22} (Theorem~2.1 and Remark~2.1).

\begin{theorem}\label{pr-pos} Assume that a locally compact space\/ $X$ is\/ $\sigma$-compact,
and that a\/ {\rm(}strictly positive definite\/{\rm)} kernel\/ $\kappa$ satisfies the first and second maximum principles.
For any\/ $\mu,\nu\in\mathcal E^+$ with\/
$U^\mu\leqslant U^\nu$ n.e.\ on\/ $X$, we have\/ $\mu(X)\leqslant\nu(X)$.
In the case where\/ $X$ is compact, this remains valid with the second maximum principle dropped.
\end{theorem}

Unless explicitly stated otherwise, from now on a (strictly positive definite) kernel $\kappa$ will always be assumed to satisfy {\it the consistency principle}, or equivalently to be {\it perfect}, which means that the cone $\mathcal E^+$ is complete in the induced strong topology, and moreover that the strong topology on $\mathcal E^+$ is finer than the (induced) vague topology on $\mathcal E^+$ \cite[Section~3.3]{F1}. Thus, for a perfect kernel $\kappa$, any strong Cauchy net $(\mu_s)\subset\mathcal E^+$ converges both strongly and vaguely to the same unique limit $\mu_0\in\mathcal E^+$, the strong and the vague topologies on $\mathcal E^+$ being Hausdorff.\footnote{Since the space $\mathfrak M$ equipped with the vague topology does not necessarily satisfy the first axiom of countability, the vague convergence cannot in general be described in terms of sequences. We follow Moore and Smith's theory of convergence
\cite{MSm}, based on the concept of {\it nets}. However, if a locally compact space $X$ is sec\-ond-count\-able, then the space $\mathfrak M$ is first-count\-able \cite[Lemma~4.4]{Z-arx}, and the use of nets in $\mathfrak M$ may often be avoided.}

\begin{remark}
As seen from a well-known counterexample by Cartan \cite{Ca1}, the whole pre-Hilbert space $\mathcal E$ is in general strongly incomplete, and this is the case even for the  Newtonian kernel $|x-y|^{2-n}$ on $\mathbb R^n$, $n\geqslant3$, despite the fact that the  Newtonian kernel is perfect (cf.\ also \cite[Theorems~1.18, 1.19]{L}).
\end{remark}

\begin{example}\label{rem:clas} The $\alpha$-Riesz kernel $|x-y|^{\alpha-n}$ of order $\alpha\in(0,2]$, $\alpha<n$, on $\mathbb R^n$, $n\geqslant2$ (thus in particular the Newtonian kernel $|x-y|^{2-n}$ on $\mathbb R^n$, $n\geqslant3$), is perfect, and it satisfies the first and second maximum principles \cite[Theorems~1.10, 1.15, 1.18, 1.27, 1.29]{L}. The same holds true for the associated $\alpha$-Green kernel on an arbitrary open subset of $\mathbb R^n$, $n\geqslant2$ \cite[Theorems~4.6, 4.9, 4.11]{FZ}. The ($2$-)Green kernel on a planar Greenian set is likewise strictly positive definite \cite[Section~I.XIII.7]{Doob} and perfect \cite{E}, and it fulfills the first and the second maximum principles (see \cite[Theorem~5.1.11]{AG} or \cite[Section~I.V.10]{Doob}). The restriction of the logarithmic kernel $-\log\,|x-y|$ to a closed disc in $\mathbb R^2$ of radius ${}<1$ satisfies the energy principle as well as Frostman's maximum principle \cite[Theorems~1.6, 1.16]{L}, and hence it is perfect \cite[Theorem~3.4.2]{F1}. (However, the domination principle then fails in general; it does hold only in a weaker sense where $\mu,\nu$ involved in the ab\-ove-quo\-ted definition
meet the additional requirement that $\nu(\mathbb R^2)\leqslant\mu(\mathbb R^2)$, cf.\ \cite[Theorem~3.2]{ST}.) Thus all the above-mentioned kernels also satisfy the principle of positivity of mass (Theorem~\ref{pr-pos}).
Also note that the $\alpha$-Riesz kernels of order $2<\alpha<n$ on $\mathbb R^n$, $n\geqslant2$, are likewise perfect \cite[Theorems~1.15, 1.18]{L},\footnote{However, these kernels do not satisfy any of the two maximum principles, since the $\alpha$-Riesz potentials $U^\mu$, where $2<\alpha<n$ and $\mu\geqslant0$, are superharmonic on $\mathbb R^n$ (see \cite[Theorems~1.4, 1.5]{L}).} and so are the Deny kernels, defined with the aid of Fourier transformation (see condition $(A)$ in \cite[Section~1]{De2}, cf.\ \cite[Section~VI.1.2]{L}).\end{example}

\subsection{Quasiclosed sets}\label{def-quasi} For arbitrary $A\subset X$, denote by $\mathcal E'(A)$ the closure of the cone $\mathcal E^+(A)$ in the strong topology on $\mathcal E^+$. Being a strongly closed subset of the strongly complete cone $\mathcal E^+$, $\mathcal E'(A)$ is likewise strongly complete. Furthermore,
\begin{equation}\label{belong}
\mathcal E'(A)\subset\mathcal E^+(\overline{A}),
\end{equation}
for $\mathcal E^+(\overline{A})$ is strongly closed (see Theorem~\ref{l-quasi} below). Here $\overline{A}:={\rm Cl}_XA$.

\begin{definition}[{\rm Fuglede \cite{F71}}]
A set $A\subset X$ is said to be {\it quasiclosed\/} if
\begin{equation*}\label{def-q}
\inf\,\bigl\{{\rm cap}^*(A\bigtriangleup F):\ F\text{ closed, }F\subset X\bigr\}=0,
\end{equation*}
$\bigtriangleup$ being the symmetric difference. Replacing here "closed" by "compact", we arrive at the concept of a {\it quasicompact\/} set.
\end{definition}

\begin{theorem}\label{l-quasi}If\/ $A\subset X$ is quasiclosed\/ {\rm(}or in particular quasicompact\/{\rm)}, then the cone\/ $\mathcal E^+(A)$ is strongly closed, and hence
\begin{equation*}
\mathcal E'(A)=\mathcal E^+(A).
\end{equation*}
\end{theorem}

\begin{proof}
Given a net $(\mu_s)\subset\mathcal E^+(A)$ converging strongly (hence vaguely) to $\mu_0\in\mathcal E^+$, we only need to show that $\mu_0$ is concentrated on $A$. For closed $A$, the cone $\mathfrak M^+(A)$ is vaguely closed according to \cite[Section~III.2, Proposition~6]{B2}, whence $\mu_0\in\mathfrak M^+(A)$.

For quasiclosed $A$, note that for every $q\in(0,\infty)$, $\mathcal E^+_q:=\{\mu\in\mathcal E^+:\ \|\mu\|\leqslant q\}$
is hereditary \cite[Definition~5.2]{Fu4} and vaguely compact \cite[Lemma~2.5.1]{F1}.
Since $(\mu_s)$ can certainly be chosen to be strongly bounded, $(\mu_s)\subset\mathcal E^+_{q'}$ for some $q'\in(0,\infty)$. Applying now \cite[Corollary~6.2]{Fu4} with $\mathcal J:=\mathcal E^+_{q'}$ and $H:=A$, we infer that $\widetilde{\mathcal J}:=\mathcal E^+_{q'}\cap\mathfrak M^+(A)$ is vaguely compact, and consequently, $(\mu_s)$ $({}\subset\widetilde{\mathcal J})$ has a vague limit point $\nu_0\in\widetilde{\mathcal J}$. The vague topology being Hausdorff, $\nu_0=\mu_0$, whence $\mu_0\in\widetilde{\mathcal J}\subset\mathfrak M^+(A)$ as desired.\end{proof}

\subsection{Inner capacitary measures}\label{sec-cap}
Assume for a moment that a set $A=K\subset X$ is compact, and $w(K)<\infty$; then the infimum in (\ref{W}) is an actual minimum, the energy being vaguely l.s.c.\ on $\mathfrak M^+$ (the principle of descent), whereas the class $\breve{\mathfrak M}^+(K)$ being vaguely compact (cf.\ \cite[Section~III.1, Corollary~3]{B2}).
But if $A\subset X$ is {\it noncompact}, the class $\breve{\mathfrak M}^+(A)$ is no longer vaguely compact, and the infimum in (\ref{W}) is in general not attained among the admissible measures --- not even for perfect kernels.

To overcome this inconvenience, Fuglede \cite[Theorem~4.1]{F1} has formulated a {\it dual\/} minimum energy problem, which leads to the same concept of inner capacity, but it is already solvable. Under the assumptions of the present study, where the kernel $\kappa$ is assumed throughout to be perfect, Fuglede's result can be specified as follows.

\begin{theorem}[{\rm\cite[Theorem~6.1]{Z-arx-22}}]\label{prop.1.2'}For any\/ $A\subset X$,
\begin{equation}\label{I}\inf_{\nu\in\Gamma_A}\,\|\nu\|^2={\rm cap}_*(A),\end{equation}
where
\begin{equation*}\Gamma_A:=\bigl\{\nu\in\mathcal E^+: \ U^\nu\geqslant1\text{ \ n.e.\ on $A$}\bigr\}.\end{equation*}
If\/ ${\rm cap}_*(A)<\infty$, the infimum in\/ {\rm(\ref{I})} is an actual minimum with the unique minimizer\/ $\gamma_A\in\Gamma_A$, called the inner capacitary measure of\/ $A$ and having the properties
\begin{gather}
\label{pr1}\gamma_A(X)=\|\gamma_A\|^2={\rm cap}_*(A),\\
\label{pr0}U^{\gamma_A}\geqslant1\text{ \ n.e.\ on\/ $A$},\\
\notag U^{\gamma_A}\leqslant1\text{ \ on\/ $S(\gamma_A)$},\\
\notag U^{\gamma_A}=1\text{ \ $\gamma_A$-a.e.\ on\/ $X$.}
\end{gather}
If moreover Frostman's maximum principle is fulfilled, then also
\begin{equation}\label{frr}
U^{\gamma_A}=1\text{ \ n.e.\ on $A$.}
\end{equation}
\end{theorem}

The following convergence theorem is often useful in applications.

 \begin{theorem}[{\rm\cite[Theorem~8.1]{Z-arx-22}}]\label{cor2}For any\/ $A\subset X$ with\/ ${\rm cap}_*(A)<\infty$,
\begin{equation*}\gamma_K\to\gamma_A\text{ \ strongly and vaguely in\/ $\mathcal E^+$ as\/ $K\uparrow A$.}\end{equation*}
If moreover the first and the second maximum principles both hold,
then also
\begin{equation*}U^{\gamma_K}\uparrow U^{\gamma_A}\text{ \ pointwise on\/ $X$ as\/ $K\uparrow A$.}\end{equation*}
\end{theorem}

In general, $\gamma_A$ is {\it not\/} concentrated on the set $A$ itself.\footnote{For instance, if $A:=B_r$ is an open ball $\{|x|<r\}$, $r>0$, in $\mathbb R^n$, $n\geqslant3$, and if $\kappa(x,y)$ is the Newtonian kernel $|x-y|^{2-n}$, then $\gamma_{B_r}$ is the positive measure of total mass $r^{n-2}$, uniformly distributed over the sphere $\{|x|=r\}$. Thus $S(\gamma_{B_r})\cap B_r=\varnothing$. (Compare with Theorem~\ref{prop.1.2''}.) \label{in-f}}
However, $\gamma_A\in\mathcal E^+(A)$ would necessarily hold if $\mathcal E^+(A)$ were assumed to be strongly closed. (Indeed, then $\gamma_A\in\mathcal E'(A)=\mathcal E^+(A)$, the former relation being implied by Theorem~\ref{cor2}.)

\begin{theorem}[{\rm cf.\ \cite[Theorem~7.2]{Z-arx-22}}]\label{prop.1.2''} Given\/ $A\subset X$ with\/ ${\rm cap}_*(A)<\infty$, assume that the class\/ $\mathcal E^+(A)$ is strongly closed\/ {\rm(}or in particular that the set\/ $A$ is quasiclosed\/{\rm)}. Then necessarily\/ $\gamma_A\in\mathcal E^+(A)$. If moreover Frostman's maximum principle is fulfilled, then\/ $\gamma_A$ is uniquely characterized within\/ $\mathcal E^+(A)$ by property\/ {\rm(\ref{frr})}.\footnote{We emphasize that, if the requirement of the strong closedness of the class $\mathcal E^+(A)$ is dropped, then there is in general {\it no\/} $\nu\in\mathcal E^+(A)$ with $U^\nu=1$ n.e.\ on $A$ (see \cite[footnote~3]{Z-arx-22}).}\end{theorem}

For other alternative characterizations of the inner capacitary measure $\gamma_A$ which hold true even for {\it arbitrary\/} $A$, we refer to \cite[Theorems~9.1--9.3]{Z-arx-22}.

\subsection{Inner balayage}\label{sec-bal} Throughout this subsection, a (perfect) kernel $\kappa$ is additionally required to satisfy {\it the domination principle}.

Assume for a moment that $A=K\subset X$ is {\it compact}.  Then one can easily prove by generalizing the classical Gauss variational method (see \cite{Ca2}, cf.\ also \cite[Section~IV.5.23]{L}) that for any given $\mu\in\mathcal E^+$, there exists $\mu^K\in\mathcal E^+(K)$ uniquely determined within $\mathcal E^+(K)$ by the equality $U^{\mu^K}=U^\mu$ n.e.\ on $K$, the uniqueness being obvious from (\ref{E}) applied to the space $X:=K$ and the kernel $\kappa':=\kappa|_{K\times K}$.
This $\mu^K$ is said to be the {\it balayage\/} of $\mu\in\mathcal E^+$ onto (compact) $K$.

If now $A\subset X$ is {\it arbitrary}, there is in general no $\nu\in\mathcal E^+(A)$ having the property
$U^\nu=U^\mu$ n.e.\ on $A$.
Nevertheless, a substantial theory of inner balayage (sweeping),
generalizing Cartan's theory \cite{Ca2} of Newtonian balayage to a suitable kernel $\kappa$ on a locally compact space $X$,
was developed (see \cite{Z-arx1}--\cite{Z-arx}), and this was performed by means of several alternative approaches described in Theorem~\ref{th-intr} below.

For any $\mu\in\mathcal E^+$ and any $A\subset X$, denote
\[\Gamma_{A,\mu}:=\bigl\{\nu\in\mathcal E^+: \ U^\nu\geqslant U^\mu\text{ \ n.e.\ on $A$}\bigr\}.\]

\begin{theorem}[{\rm\cite[Theorem~3.1]{Z-arx-22}}]\label{th-intr} There exists precisely the same unique measure\/ $\mu^A\in\Gamma_{A,\mu}\cap\mathcal E'(A)$, called the inner balayage of\/ $\mu\in\mathcal E^+$ to\/ $A\subset X$, and solving each of the following three extremal problems:
\begin{gather}
\notag
 \|\mu^A\|=\min_{\nu\in\Gamma_{A,\mu}}\,\|\nu\|,\\
\notag U^{\mu^A}=\min_{\nu\in\Gamma_{A,\mu}}\,U^\nu\text{ \ on\ $X$},\\
\label{opr'}
\|\mu-\mu^A\|=\min_{\nu\in\mathcal E'(A)}\,\|\mu-\nu\|.
\end{gather}
This\/ $\mu^A$ is said to be the inner balayage of\/ $\mu$ to\/ $A$, and it has the properties
\begin{align}
\label{eq-pr-10}U^{\mu^A}&=U^\mu\text{ \ n.e.\ on\ }A,\\
\label{eq-pr-11}U^{\mu^A}&=U^\mu\text{ \ $\mu^A$-a.e.,}\\
\notag U^{\mu^A}&\leqslant U^\mu\text{ \ on $X$.}
\end{align}
Moreover, {\rm(\ref{eq-pr-10})} characterizes the inner balayage\/ $\mu^A$ uniquely within\/ $\mathcal E'(A)$.
\end{theorem}

Thus, according to (\ref{opr'}), the inner balayage $\mu^A$ can in particular be characterized as the (unique) {\it orthogonal projection\/} of $\mu\in\mathcal E^+$ in the pre-Hil\-bert space $\mathcal E$ onto the (strongly complete, convex) cone $\mathcal E'(A)$, cf.\ \cite[Theorem~1.12.3]{E2}.

\begin{corollary}\label{eqba}
 For arbitrary\/ $A$ with\/ ${\rm cap}_*(A)<\infty$,
 \begin{equation}\label{opr''}\gamma_A=(\gamma_A)^A.\end{equation}
\end{corollary}

\begin{proof}
Since $\gamma_A\in\mathcal E'(A)$ (cf.\ Section~\ref{sec-cap}), whereas the orthogonal projection of any $\nu\in\mathcal E'(A)$ onto $\mathcal E'(A)$ is certainly the same $\nu$, (\ref{opr'}) applied to $\gamma_A$ gives (\ref{opr''}).
\end{proof}

In general, $\mu^A\not\in\mathcal E^+(A)$. (Indeed, for $A:=B_r\subset\mathbb R^n$, $\kappa(x,y):=|x-y|^{2-n}$, and $\mu:=\gamma_{B_r}$, where  $n\geqslant3$ and $r\in(0,\infty)$, this is obvious from Corollary~\ref{eqba} and footnote~\ref{in-f}.) Nonetheless, the following assertion holds true.

\begin{corollary}\label{cor-bal}
 If\/ $\mathcal E^+(A)$ is strongly closed\/ {\rm(}or in particular if\/ $A$ is quasiclosed\/{\rm)}, then the inner balayage\/ $\mu^A$ can alternatively be characterized as the orthogonal projection of\/ $\mu$ in the pre-Hilbert space\/ $\mathcal E$ onto the cone\/ $\mathcal E^+(A)$; i.e.\ $\mu^A\in\mathcal E^+(A)$ and
\begin{equation*}
\|\mu-\mu^A\|=\min_{\nu\in\mathcal E^+(A)}\,\|\mu-\nu\|.
\end{equation*}
The same\/ $\mu^A$ is uniquely determined within\/ $\mathcal E^+(A)$ by property\/ {\rm(\ref{eq-pr-10})}.
\end{corollary}

\begin{proof}
 Since then $\mathcal E'(A)=\mathcal E^+(A)$,  this is obvious from Theorem~\ref{th-intr}.
\end{proof}

The following Theorem~\ref{th-bal-cont} justifies the term "{\it inner\/} balayage". (It is worth emphasizing here that, although a kernel $\kappa$ was assumed in \cite{Z-arx1} to be positive, this restriction on $\kappa$ is unnecessary for the validity of the results from \cite{Z-arx1} quoted below.)

\begin{theorem}[{\rm\cite[Theorem~4.8]{Z-arx1}}]\label{th-bal-cont}For any\/ $\mu\in\mathcal E^+$ and any\/ $A\subset X$,
\begin{gather*}\mu^K\to\mu^A\text{ \ strongly and vaguely in\/ $\mathcal E^+$ as\/ $K\uparrow A$},\\
U^{\mu^K}\uparrow U^{\mu^A}\text{ \ pointwise on\/ $X$ as\/ $K\uparrow A$}.\end{gather*}
\end{theorem}

We complete this brief review of the theory of inner balayage with the following useful properties of inner swept measures (see  \cite[Propositions~7.2, 7.4]{Z-arx1}, cf.\ \cite{Z-arx1-err}).

\begin{proposition}[{\rm Balayage "with a rest"}\footnote{Regarding the terminology used here, we follow N.S.~Landkof \cite[p.~264]{L}.}
]\label{cor-rest}If\/ $A\subset Q$, then for any\/ $\mu\in\mathcal E^+$,
\begin{equation*}\label{eq-rest}\mu^A=(\mu^Q)^A.\end{equation*}
\end{proposition}

\begin{proposition}\label{cor-mass} If Frostman's maximum principle is fulfilled, then\/\footnote{It is worth noting that, if moreover $\kappa\geqslant0$ and ${\rm cap}_*(A)<\infty$, or if $A$ is compact, then actually $\mu^A(X)=\int U^\mu\,d\gamma_A$. See \cite{Z-arx1} (Eq.~(7.3) and Proposition~7.6), cf.\ \cite{Z-arx1-err}.}
\begin{equation}\label{eq-mass}\mu^A(X)\leqslant\mu(X)\text{ \ for any\/ $\mu\in\mathcal E^+$ and\/ $A\subset X$}.\end{equation}
\end{proposition}

\section{Preliminary results on Problem~\ref{pr-main}}\label{sec-ext}

Keeping the (permanent) assumption of the energy principle, in the present section we do not require a kernel $\kappa$ to be perfect.

Fix an external field $f$ admitting the
representation\footnote{In particular, (\ref{ff}) necessarily holds for $f$ of form (\ref{fform}). In fact, in view of (\ref{phi}) one can take $f_1:=\psi+U^{\vartheta^+}$ and $f_2:=U^{\vartheta^-}$, where $\vartheta^+$ and $\vartheta^-$ denote the positive and negative parts of $\vartheta\in\mathcal E$ in the Hahn--Jor\-dan decomposition. The required inequality $\int U^{\vartheta^-}\,d\mu<\infty$ is indeed fulfilled for all $\mu\in\mathcal E^+$ since, by virtue of the positive definiteness of the kernel, $\mathcal E=\mathcal E^+-\mathcal E^+$ (see \cite[Section~3.1]{F1}).\label{FFF}}
\begin{equation}\label{ff}
 f=f_1-f_2,\text{ \ where $f_1,f_2\in\Phi(X)$ and $\int f_2\,d\mu<\infty$ for all $\mu\in\mathcal E^+$.}
\end{equation}
Observing with the aid of Lemma~\ref{l-negl1} that then necessarily $f_2<\infty$ n.e.\ on $X$, we infer that such $f$ is well defined and takes values in $(-\infty,\infty]$ n.e.\ on $X$. Also note that $f$ is {\it universally measurable}, i.e.\ $\nu$-meas\-ur\-able for all $\nu\in\mathfrak M^+$.

For any $\mu\in\mathcal E^+$, define the {\it $f$-wei\-ghted potential\/} $U_f^\mu$ by means of the formula
\[U_f^\mu:=U^\mu+f;\]
then $U_f^\mu$ is well defined and takes values in $(-\infty,\infty]$ n.e.\ on $X$. This follows from the fact that the same holds true for both $U^\mu$ and $f$, by making use of the countable subadditivity of inner capacity on universally measurable sets (Lemma~\ref{str}).

Let $\mathcal E^+_f$ stand for the (convex) class of all $\mu\in\mathcal E^+$ such that the function $f$ (or, equivalently, $f_1$) is $\mu$-in\-t\-eg\-r\-able. Then for every $\mu\in\mathcal E^+_f$, the $f$-wei\-ghted energy $I_f(\mu)$, introduced by (\ref{fen}), is finite and, by Lebesgue--Fubini's theorem \cite[Section~V.8, Theorem~1]{B2}, is representable in the form
\[I_f(\mu)=\int\bigl(U_f^\mu+f\bigr)\,d\mu\in(-\infty,\infty).\]

For arbitrary $A\subset X$, denote
\[\mathcal E^+_f(A):=\mathcal E^+_f\cap\mathfrak M^+(A),\quad\breve{\mathcal E}^+_f(A):=\mathcal E^+_f\cap\breve{\mathfrak M}^+(A),\] and let $w_f(A)$ stand for the infimum of $I_f(\mu)$, where $\mu$ ranges over $\breve{\mathcal E}^+_f(A)$, see (\ref{wf}).
If this infimum is finite, see (\ref{fin}),\footnote{It follows immediately from (\ref{fin}) that ${\rm cap}_*(A)>0$ (for if not, $\breve{\mathcal E}^+(A)=\varnothing$ by Lemma~\ref{l-negl1}). Actually, even a stronger assertion then necessarily holds --- see Lemma~\ref{l-aux-2}.}
one can consider the inner Gauss variational problem on the existence of  $\lambda=\lambda_{A,f}\in\breve{\mathcal E}^+_f(A)$ with $I_f(\lambda_{A,f})=w_f(A)$, see Problem~\ref{pr-main}.

The present section contains preliminary results on Problem~\ref{pr-main}, some of them being established in a much more general form in the author's earlier paper \cite{Z5a}.

To begin with, we note that {\it a solution\/ $\lambda_{A,f}$ to Problem\/~{\rm\ref{pr-main}} is unique\/} (if it exists), which follows easily from the convexity of the class $\breve{\mathcal E}^+_f(A)$ and the energy principle, by use of the parallelogram identity in the pre-Hil\-bert space $\mathcal E$  (cf.\ \cite[Lemma~6]{Z5a}).

\begin{lemma}[{\rm cf.\ \cite[Lemma~4]{Z5a}}]\label{l-aux-1} Given\/ $\kappa$, $A$, and\/ $f$,
\begin{equation}\label{l-mon}
w_f(K)\downarrow w_f(A)\text{ \ as\/ $K\uparrow A$}.
\end{equation}
\end{lemma}

In fact, (\ref{l-mon}) can be derived from the relation
\begin{equation}\label{contf}
I_f(\mu)=\lim_{K\uparrow A}\,I_f(\mu|_K)\text{ \ for all $\mu\in\mathcal E^+_f(A)$},
\end{equation}
which in turn follows by applying Lemma~\ref{l-lower} to each of the $\mu$-integrable functions $\kappa$, $f_1$, and $f_2$ (cf.\ Remark~\ref{fIII'}), $f_1$ and $f_2$ appearing in (\ref{ff}). Formula (\ref{contf}) also implies that the extremal value
$w_f(A)$ would be the same if the admissible measures $\mu$ in (\ref{wf}) were required to be of compact support $S(\mu)\subset A$ (cf.\ \cite[Remark~2]{Z5a}).

\begin{lemma}[{\rm cf.\ \cite[Lemma~5]{Z5a}}]\label{l-aux-2} $w_f(A)<\infty$ holds if and only if
\begin{equation}\label{iv}
{\rm cap}_*\bigl(\{x\in A:\ f_1(x)<\infty\}\bigr)>0.
\end{equation}
\end{lemma}

\begin{lemma}\label{l-aux-3}Assume an external field\/ $f$ is of form\/ {\rm(\ref{fform})}, that is, $f=\psi+U^\vartheta$, where\/ $\psi\in\Phi(X)$ and\/ $\vartheta\in\mathcal E$. Then\/ {\rm(\ref{fin})} is equivalent to\/ $w_f(A)<\infty$, and hence to
\begin{equation}\label{iv;}
{\rm cap}_*\bigl(\{x\in A:\ \psi(x)<\infty\}\bigr)>0.
\end{equation}
\end{lemma}

\begin{proof}The first part of the claim will follow once we show that
\begin{equation*}\label{-infty}
 w_f(A)>-\infty.
\end{equation*}
For any given $\vartheta\in\mathcal E$,
\begin{equation}\label{Isigma}
I_{U^\vartheta}(\mu)=\|\mu\|^2+2\int U^\vartheta\,d\mu=\|\mu+\vartheta\|^2-\|\vartheta\|^2\text{ \ for all $\mu\in\mathcal E^+$},
\end{equation}
and hence, by the energy principle,
\[w_{U^\vartheta}(A)\geqslant-\|\vartheta\|^2>-\infty.\]
It thus remains to prove that for any given $\psi\in\Phi(X)$,
\[\inf_{\mu\in\breve{\mathcal E}^+(A)}\,\int\psi\,d\mu>-\infty.\]
This however is obvious if $\psi\geqslant0$, while the remaining case of compact $X$ is treated as described in Remark~\ref{fIII'}, by use of the equality $\mu(X)=1$ for all $\mu\in\breve{\mathcal E}^+(A)$.\footnote{\label{Wfin}
The above proof also implies the following observation, to be useful in the sequel: {\it if an external field\/ $f$ is of form\/ {\rm(\ref{fform})}, then\/
$I_f(\mu)>-\infty$ for all\/ $\mu\in\mathcal E^+$.}}

In view of Lemma~\ref{l-aux-2}, the latter part of the claim is reduced to the equivalence of (\ref{iv}) and (\ref{iv;}), which however is obvious from Lemma~\ref{str} and the fact that the potential of a measure of finite energy is finite q.e.\ on $X$ (cf.\ footnote~\ref{FFF}).\end{proof}

\begin{corollary}\label{l-aux-4}Assume that\/ $f=U^\vartheta$, where\/ $\vartheta\in\mathcal E$. Then
\begin{equation}\label{Isigma'}
 \mathcal E^+_f(A)=\mathcal E^+(A).
\end{equation}
Furthermore, assumption\/ {\rm(\ref{fin})} is fulfilled if and only if\/ ${\rm cap}_*(A)>0$.
\end{corollary}

\begin{proof}
Noting from (\ref{Isigma}) that $I_f(\mu)$ is finite for all $\mu\in\mathcal E^+$, we obtain (\ref{Isigma'}), while
 the latter part of the claim is an immediate consequence of Lemma~\ref{l-aux-3} with $\psi=0$.
\end{proof}

\subsection{Characteristic properties of $\lambda_{A,f}$}\label{sec-ch} In Theorems~\ref{th-ch1} and \ref{th-ch2}, assume again an external field $f$ to be of form (\ref{ff}). Also assume (\ref{fin}) to be fulfilled (see Lemmas~\ref{l-aux-2}, \ref{l-aux-3} and Corollary~\ref{l-aux-4} for necessary and/or sufficient conditions for this to hold).

\begin{theorem}[{\rm cf.\ \cite[Theorems~1, 2]{Z5a}}]\label{th-ch1}
If the solution\/ $\lambda=\lambda_{A,f}$ to Problem\/~{\rm\ref{pr-main}} exists, then its\/ $f$-wei\-ghted potential\/ $U_f^{\lambda}$ has the properties
\begin{align}
  U_f^{\lambda}&\geqslant c_{A,f}\text{ \ n.e.\ on\/ $A$},\label{t1-pr1}\\
  U_f^{\lambda}&=c_{A,f}\text{ \ $\lambda$-a.e.\ on\/ $X$},\label{t1-pr2}
\end{align}
where
\begin{equation}\label{cc}
c_{A,f}:=\int U_f^{\lambda}\,d\lambda=w_f(A)-\int f\,d\lambda\in(-\infty,\infty)
\end{equation}
is said to be the inner\/ $f$-weighted equilibrium constant for\/ $A$.

If moreover\/ $f\in\Phi(X)$, then also
\[U_f^{\lambda}\leqslant c_{A,f}\text{ \ on\/ $S(\lambda)$},\]
and hence
\[U_f^{\lambda}=c_{A,f}\text{ \ n.e.\ on\/ $S(\lambda)\cap A$}.\]
\end{theorem}

Relations (\ref{t1-pr1}) and (\ref{cc}), resp.\ (\ref{t1-pr2}) and (\ref{cc}), characterize the minimizer $\lambda_{A,f}$ uniquely within $\breve{\mathcal E}^+_f(A)$. In more detail, the following theorem holds true.

\begin{theorem}[{\rm cf.\ \cite[Proposition~1]{Z5a}}]\label{th-ch2}For\/ $\mu\in\breve{\mathcal E}^+_f(A)$ to be the\/ {\rm(}unique\/{\rm)} solution\/ $\lambda_{A,f}$ to Problem\/~{\rm\ref{pr-main}}, it is necessary and sufficient that either of the following two characteristic inequalities be fulfilled:
\begin{align}U_f^\mu&\geqslant\int U_f^\mu\,d\mu=:c_1\text{ n.e.\ on\/ $A$},\label{1}\\
U_f^\mu&\leqslant w_f(A)-\int f\,d\mu=:c_2\text{ \ $\mu$-a.e.\ on\/ $X$.}\label{2}
\end{align}
If\/ {\rm(\ref{1})} or\/ {\rm(\ref{2})} holds true, then equality actually prevails in\/ {\rm(\ref{2})}, and moreover
\[c_1=c_2=c_{A,f},\] the inner\/ $f$-weighted equilibrium constant\/ $c_{A,f}$ being introduced by formula\/ {\rm(\ref{cc})}.
\end{theorem}

\section{Proofs of the main results}\label{sec-proofs}

Throughout this section,
\begin{equation}\label{wff}
w_f(A)<\infty,
\end{equation}
and so the class $\breve{\mathcal E}^+_f(A)$ of admissible measures in Problem~\ref{pr-main} is nonempty.

Assume for a moment that $A=K\subset X$ is compact, and that $f\in\Phi(X)$. Then Problem~\ref{pr-main} is solvable for an arbitrary (not necessarily perfect) kernel $\kappa$ on $X$, which follows from the vague compactness of the class $\breve{\mathfrak M}^+(K)$ \cite[Section~III.1, Corollary~3]{B2} and the vague lower semicontinuity of the $f$-wei\-ghted energy $I_f(\mu)$ on $\mathfrak M^+$, the latter being obvious from Lemma~\ref{lemma-semi} applied to each of $\kappa\in\Phi(X\times X)$ and $f\in\Phi(X)$ (cf.\ also Remark~\ref{fIII'}). However, this fails to hold if either of the above two requirements is omitted, and moreover Problem~\ref{pr-main} then becomes in general unsolvable.

To analyze Problem~\ref{pr-main} for quite a general (not necessarily of the class $\Phi(X)$) external field $f$ and for quite a general (not necessarily closed) set $A\subset X$, {\it from now on we assume that the kernel\/ $\kappa$ is perfect, and that\/ $f$ is of form\/} (\ref{fform}), i.e.
\[f=\psi+U^\vartheta,\text{ \ where $\psi\in\Phi(X)$ and $\vartheta\in\mathcal E$}\]
(see assumptions $(H_1)$ and $(H_2)$ in Section~\ref{sec-intr}). This enables us to develop an approach based on the close interaction between the strong and the vague topologies on the (strongly complete) cone $\mathcal E^+$, which turned out to be efficient for proving Theorems~\ref{th2'}, \ref{th4'}, \ref{th3'}, \ref{th5'}, \ref{riesz} and Corollary~\ref{cortm} (see Sections~\ref{sec-pr}--\ref{sec-pr6}).

Our choice of $f$ is basically caused by the fact that, if a net $(\mu_s)_{s\in S}\subset\mathcal E^+$ converges to $\mu_0\in\mathcal E^+$ both strongly and vaguely, then (see the proof of Lemma~\ref{la2})
\[I_f(\mu_0)\leqslant\liminf_{s\in S}\,I_f(\mu_s),\]
which is one of the crucial points in the analysis performed below.

\subsection{Preparatory lemmas}\label{sec-lemmas}
A net $(\mu_s)_{s\in S}\subset\breve{\mathcal E}^+_f(A)$ is said to be {\it minimizing\/} (in Problem~\ref{pr-main}) if
\begin{equation}\label{min}
\lim_{s\in S}\,I_f(\mu_s)=w_f(A).
\end{equation}
The (nonempty) set of all those $(\mu_s)_{s\in S}$ is denoted by $\mathbb M_f(A)$.

\begin{lemma}\label{la1}
  There is the unique\/ $\xi_{A,f}\in\mathcal E^+$ such that, for all\/ $(\mu_s)_{s\in S}\in\mathbb M_f(A)$,
  \begin{equation}\label{conv}
    \mu_s\to\xi_{A,f}\text{ \ strongly and vaguely}.
  \end{equation}
This\/ $\xi_{A,f}$ is said to be the extremal measure\/ {\rm(}in Problem\/~{\rm\ref{pr-main}}{\rm)}.
\end{lemma}

\begin{proof} We shall first show that for any $(\mu_s)_{s\in S}$ and $(\nu_t)_{t\in T}$
from $\mathbb M_f(A)$,
\begin{equation}\label{ST}
\lim_{(s,t)\in S\times T}\,\|\mu_s-\nu_t\|=0,
\end{equation}
$S\times T$ being the directed product of the directed sets $S$ and
$T$ (see e.g.\ \cite[p.~68]{K}). In fact, due to the convexity of the class $\breve{\mathcal E}^+_f(A)$, for any $(s,t)\in S\times T$ we have
\[4w_f(A)\leqslant4I_f\Bigl(\frac{\mu_s+\nu_t}{2}\Bigr)=\|\mu_s+\nu_t\|^2+4\int f\,d(\mu_s+\nu_t).\]
Applying the parallelogram identity in the pre-Hilbert space $\mathcal E$ therefore gives
\[0\leqslant\|\mu_s-\nu_t\|^2\leqslant-4w_f(A)+2I_f(\mu_s)+2I_f(\nu_t),\]
which yields (\ref{ST}) when combined with (\ref{min}).

Taking the two nets in (\ref{ST}) to be equal, we see that every $(\nu_t)_{t\in T}\in\mathbb M_f(A)$ is strong Cauchy. The cone $\mathcal E^+$ being strongly complete, $(\nu_t)_{t\in T}$ must converge strongly to some $\xi_{A,f}\in\mathcal E^+$. The same unique $\xi_{A,f}$ also serves as the strong limit of any other $(\mu_s)_{s\in S}\in\mathbb M_f(A)$, which is obvious from (\ref{ST}). The strong topology on $\mathcal E^+$ being finer than the vague topology on $\mathcal E^+$ by virtue of the perfectness of the kernel, $(\mu_s)_{s\in S}$ must converge to $\xi_{A,f}$ also vaguely.
\end{proof}

\begin{remark}\label{tmass}In general, the extremal measure $\xi_{A,f}$ is {\it not\/} concentrated on $A$. What is clear so far is that
\begin{equation}\label{Eq}
\xi_{A,f}\in\mathcal E'(A)\subset\mathcal E^+(\overline{A}),
\end{equation}
the former relation being clear from (\ref{conv}), and the latter from (\ref{belong}).\end{remark}

\begin{remark}\label{tmass'}Another consequence of (\ref{conv}) is that
\begin{equation}\label{11}
\xi_{A,f}(X)\leqslant1,
\end{equation}
the map $\mu\mapsto\mu(X)$ being vaguely l.s.c.\ on $\mathfrak M^+$ (Lemma~\ref{lemma-semi} with $g=1$). {\it Equality necessarily prevails in\/ {\rm(\ref{11})} if\/ $A=K$ is compact}, cf.\ \cite[Section~III.1, Corollary~3]{B2}.
\end{remark}

If the unweighted case $f=0$ takes place, we shall drop the index $f$, and shall simply write $\lambda_A$, $\mathbb M(A)$, and $\xi_A$ in place of $\lambda_{A,f}$, $\mathbb M_f(A)$, and $\xi_{A,f}$, respectively.

\begin{lemma}\label{la2} For the extremal measure\/ $\xi:=\xi_{A,f}$, we have
 \begin{equation}\label{IorII}
 -\infty<I_f(\xi_{A,f})\leqslant w_f(A)<\infty.
 \end{equation}
\end{lemma}

\begin{proof} The first inequality in (\ref{IorII}) is fulfilled by footnote~\ref{Wfin}, and the last by (permanent) assumption (\ref{wff}).
In view of (\ref{min}), it is therefore enough to show that
\begin{equation}\label{IorII'}
I_f(\xi)\leqslant\lim_{s\in S}\,I_f(\mu_s),
\end{equation}
$(\mu_s)_{s\in S}\in\mathbb M_f(A)$ being fixed.

Assume first that $f\in\Phi(X)$. Since $\mu_s\to\xi$ both strongly and vaguely, Lemma~\ref{lemma-semi} applied to $f$ gives (\ref{IorII'}), for (cf.\  Remark~\ref{fIII'})
\[I_f(\xi)=\|\xi\|^2+2\int f\,d\xi\leqslant\liminf_{s\in S}\,\bigl(\|\mu_s\|^2+2\int f\,d\mu_s\bigr)=\lim_{s\in S}\,I_f(\mu_s).\]

Otherwise, $f=\psi+U^\vartheta$, where $\psi\in\Phi(X)$ and $\vartheta\in\mathcal E$. Similarly as in the preceding paragraph, (\ref{IorII'}) will follow once we verify the inequality
\[\int\bigl(\psi+U^\vartheta\bigr)\,d\xi\leqslant\liminf_{s\in S}\,\int\bigl(\psi+U^\vartheta\bigr)\,d\mu_s,\]
which, again by Lemma~\ref{lemma-semi} applied to $\psi$, is reduced to the equality
\begin{equation}\label{bun}\langle\vartheta,\xi\rangle=\lim_{s\in S}\,\langle\vartheta,\mu_s\rangle.\end{equation}
Applying the Cauchy--Schwarz (Bunyakovski) inequality to the (signed) measures $\vartheta$ and $\xi-\mu_s$, $s\in S$, elements of the pre-Hil\-bert space $\mathcal E$, we get
\[\left|\langle\vartheta,\xi-\mu_s\rangle\right|\leqslant\|\vartheta\|\cdot\|\xi-\mu_s\|,\]
which by the strong convergence of $(\mu_s)_{s\in S}$ to $\xi$ proves (\ref{bun}), whence the lemma.\end{proof}

\begin{lemma}\label{l-solv}The following two assertions are equivalent:
\begin{itemize}
 \item[{\rm(i)}] There exists the solution\/ $\lambda=\lambda_{A,f}$ to Problem\/~{\rm\ref{pr-main}}.
 \item[{\rm(ii)}] The extremal measure\/ $\xi=\xi_{A,f}$ belongs to the class\/ $\breve{\mathcal E}^+(A)$.\footnote{Compare with (\ref{Eq}) and (\ref{11}).}
\end{itemize}
If either of\/ {\rm(i)} or\/ {\rm(ii)} is fulfilled, then actually\/
\begin{equation}\label{es}
 \xi_{A,f}=\lambda_{A,f}.
\end{equation}
\end{lemma}

\begin{proof} Assume (i) holds true. The trivial sequence $(\lambda)$ being obviously minimizing:
\[(\lambda)\in\mathbb M_f(A),\]
it must converge strongly (and vaguely) to the extremal measure $\xi$, see Lemma~\ref{la1}. The strong topology on $\mathcal E$  being Hausdorff, this yields (\ref{es}), whence (ii).

Assuming now that (ii) holds, we note from (\ref{IorII}) that $\xi\in\breve{\mathcal E}^+_f(A)$, whence
\[I_f(\xi)\geqslant w_f(A).\] The opposite being valid again by (\ref{IorII}), this establishes (i) with $\lambda:=\xi$.
\end{proof}

\begin{lemma}\label{la3}
For each\/ $K\in\mathfrak C_A$ large enough, there exists the\/ {\rm(}unique\/{\rm)} solution\/ $\lambda_{K,f}$ to Problem\/~{\rm\ref{pr-main}} with\/ $A:=K$, and moreover
\begin{equation}\label{eq-conv}
 \lambda_{K,f}\to\xi_{A,f}\text{ \ strongly and vaguely as\/ $K\uparrow A$}.
\end{equation}
\end{lemma}

\begin{proof}
  For each $K\in\mathfrak C_A$ large enough ($K\geqslant K_0$), we have $w_f(K)<\infty$, cf.\ (\ref{wff}) and (\ref{l-mon}). Therefore, by Lemma~\ref{la1} with $A:=K$, there is the unique $\xi_{K,f}\in\mathcal E^+$ such that every $(\mu_s)_{s\in S}\in\mathbb M_f(K)$ converges to $\xi_{K,f}$ strongly and vaguely. Noting that
  $\xi_{K,f}\in\breve{\mathcal E}^+(K)$ (see Remarks~\ref{tmass} and \ref{tmass'}), we now derive from Lemma~\ref{l-solv} that $\xi_{K,f}$ serves as the (unique) solution $\lambda_{K,f}$ to Problem~\ref{pr-main} with $A:=K$, which establishes the first part of the lemma. But $(\lambda_{K,f})_{K\geqslant K_0}\subset\breve{\mathcal E}^+_f(A)$, and moreover, by Lemma~\ref{l-aux-1},
  \[\lim_{K\uparrow A}\,I_f(\lambda_{K,f})=\lim_{K\uparrow A}\,w_f(K)=w_f(A),\]
  which shows that, actually,
  \[(\lambda_{K,f})_{K\geqslant K_0}\in\mathbb M_f(A).\] Applying Lemma~\ref{la1} once again, we obtain (\ref{eq-conv}), whence the whole lemma.
\end{proof}

\subsection{Proof of Theorem~\ref{th2'}}\label{sec-pr} Let the assumptions of the theorem be fulfilled. Due to $(H_3)$, $\mathcal E'(A)=\mathcal E^+(A)$, which substituted into (\ref{Eq}) gives
\begin{equation}\label{exm'}
\xi\in\mathcal E^+(A),
\end{equation}
$\xi=\xi_{A,f}$ being the extremal measure in Problem~\ref{pr-main} (cf.\ Lemma~\ref{la1}).

We aim to show that in the case ${\rm cap}_*(A)<\infty$, $\xi$ serves as the solution $\lambda_{A,f}$ to Problem~\ref{pr-main}. By Lemma~\ref{l-solv}, this will follow once we verify that $\xi\in\breve{\mathcal E}^+(A)$, which in view of (\ref{exm'}) is equivalent to the assertion
\begin{equation}\label{exm}\xi(X)=1.\end{equation}

Fix a minimizing net $(\mu_s)_{s\in S}\in\mathbb M_f(A)$. Taking a subnet (if necessary) and changing notation, we can certainly assume $(\mu_s)_{s\in S}$ to be strongly bounded:
\begin{equation}\label{str-b}
 \sup_{s\in S}\,\|\mu_s\|<\infty.
\end{equation}
Since $\mu_s\to\xi$ vaguely, cf.\ (\ref{conv}), applying \cite[Section~IV.4, Corollary~3]{B2} gives
\begin{equation}\label{upper}\int1_K\,d\xi\geqslant\limsup_{s\in S}\,\int1_K\,d\mu_s\text{ \ for every compact $K\subset X$},\end{equation}
the indicator function $1_K$ being bounded, of compact support, and upper semicontinuous on $X$. On the other hand,
\[\xi(X)=\lim_{K\uparrow X}\,\xi(K)=\lim_{K\uparrow X}\,\int1_K\,d\xi,\]
which together with  (\ref{11}) and (\ref{upper}) results in
\[1\geqslant\xi(X)\geqslant\limsup_{(s,K)\in S\times\mathfrak C_X}\,\int1_K\,d\mu_s=
1-\liminf_{(s,K)\in S\times\mathfrak C_X}\,\int1_{A\setminus K}\,d\mu_s,\]
the equality being implied by the fact that every $\mu_s$ is a positive measure of unit total mass concentrated on $A$. (Here $\mathfrak C_X$ is the upward directed set of all compact subsets $K$ of $X$.)
The proof of (\ref{exm}) is thus reduced to that of
\begin{equation}\label{0}
  \liminf_{(s,K)\in S\times\mathfrak C_X}\,\int1_{A\setminus K}\,d\mu_s=0.
\end{equation}

By Theorem~\ref{prop.1.2'} applied to $A\setminus K$, $K\in\mathfrak C_X$ being arbitrarily chosen, there exists the (unique) inner capacitary measure $\gamma_{A\setminus K}$, minimizing the energy $\|\mu\|^2$ over the (convex) set $\Gamma_{A\setminus K}\subset\mathcal E^+$. For any $K'\in\mathfrak C_X$ such that $K\subset K'$, we have $\Gamma_{A\setminus K}\subset\Gamma_{A\setminus K'}$, and applying \cite[Lemma~4.1.1]{F1} with $\mathcal H:=\mathcal E$ and $\Gamma:=\Gamma_{A\setminus K'}$ gives
\begin{equation}\label{str-ca}\|\gamma_{A\setminus K}-\gamma_{A\setminus K'}\|^2\leqslant\|\gamma_{A\setminus K}\|^2-\|\gamma_{A\setminus K'}\|^2.\end{equation}
Since $\|\gamma_{A\setminus K}\|^2={\rm cap}_*(A\setminus K)$ by (\ref{pr1}), $\|\gamma_{A\setminus K}\|^2$ decreases as $K$ ranges through $\mathfrak C_X$, which together with (\ref{str-ca}) implies that the net $(\gamma_{A\setminus K})_{K\in\mathfrak C_X}\subset\mathcal E^+$ is Cauchy in the strong topology on $\mathcal E^+$. Noting that $(\gamma_{A\setminus K})_{K\in\mathfrak C_X}$ converges vaguely to zero,\footnote{Indeed, for any given $\varphi\in C_0(X)$, there exists a relatively compact open set $G\subset X$ such that $\varphi(x)=0$ for all $x\not\in\overline{G}$. Hence, $\gamma_{A\setminus K}(\varphi)=0$ for all $K\in\mathfrak C_X$ with $K\supset\overline{G}$, and the claim follows.} we get
\begin{equation}\label{str-conv}
  \gamma_{A\setminus K}\to0\text{ \ strongly in $\mathcal E^+$ as $K\uparrow X$,}
\end{equation}
the kernel $\kappa$ being perfect.

But, by (\ref{pr0}) applied to $A\setminus K$,
\begin{equation}\label{ii}U^{\gamma_{A\setminus K}}\geqslant1_{A\setminus K}\text{ \ n.e.\ on $A\setminus K$},\end{equation}
hence $\mu_s$-a.e.\ for all $s\in S$, the latter being derived from Lemma~\ref{l-negl} due to the fact that $A\setminus K$ along with $A$ is $\mu_s$-mea\-sur\-ab\-le, whereas $\mu_s\in\mathcal E^+$ is bounded. Integrating (\ref{ii}) with respect to $\mu_s$ we therefore obtain, by the Cauchy--Schwarz inequality,
\[\int1_{A\setminus K}\,d\mu_s\leqslant\int U^{\gamma_{A\setminus K}}\,d\mu_s\leqslant\|\gamma_{A\setminus K}\|\cdot\|\mu_s\|\text{ \ for all $K\in\mathfrak C_X$ and $s\in S$},\]
which combined with (\ref{str-b}) and (\ref{str-conv}) results in (\ref{0}).

Thus, under the assumptions made, the solution $\lambda_{A,f}$ to Problem~\ref{pr-main} does indeed exist. The remaining part of the theorem, uniquely characterizing $\lambda_{A,f}$ within $\breve{\mathcal E}^+_f(A)$, follows immediately by applying Theorem~\ref{th-ch2}.

\subsection{Proof of Theorem~\ref{th4'}}\label{sec-pr4} Under (permanent) hypotheses $(H_1)$ and $(H_2)$, assume that the solution $\lambda_{A,f}$ to Problem~\ref{pr-main} exists; then by Lemma~\ref{l-solv}, it must coincide with the extremal measure $\xi_{A,f}$, determined by Lemma~\ref{la1}. On the other hand, according to Lemma~\ref{la3}, for every $K\in\mathfrak C_A$ large enough ($K\geqslant K_0$), there is the solution $\lambda_{K,f}$ to Problem~\ref{pr-main} with $A:=K$; and moreover the net $(\lambda_{K,f})_{K\geqslant K_0}$ converges strongly and vaguely to $\xi_{A,f}$, see (\ref{eq-conv}). Substituting $\lambda_{A,f}=\xi_{A,f}$ into (\ref{eq-conv}) we get (\ref{conv2}).

Assume now that
$f=U^\vartheta$, where $\vartheta\in\mathcal E$. By Theorem~\ref{th-ch1}, the inner $f$-wei\-ghted equilibrium constant $c_{A,f}$ can be written in the form
\[c_{A,f}=\int U_f^{\lambda_{A,f}}\,d\lambda_{A,f}=\|\lambda_{A,f}\|^2+\int U^\vartheta\,d\lambda_{A,f}=\|\lambda_{A,f}\|^2+\langle\vartheta,\lambda_{A,f}\rangle,\]
and likewise
\begin{equation*}\label{limits}
c_{K,f}=\|\lambda_{K,f}\|^2+\langle\vartheta,\lambda_{K,f}\rangle\quad(K\geqslant K_0).
\end{equation*}
Since $\lambda_{K,f}\to\lambda_{A,f}$ strongly in $\mathcal E^+$ as $K\uparrow A$, see (\ref{conv2}), applying the Cauchy--Schwarz inequality to $\vartheta$ and $\lambda_{K,f}-\lambda_{A,f}$ gives
$\langle\vartheta,\lambda_{K,f}\rangle\to\langle\vartheta,\lambda_{A,f}\rangle$ (as $K\uparrow A$),
whence (\ref{conv3}).

\subsection{Proof of Theorem~\ref{th3'}}\label{sec-pr3} Let the assumptions of the theorem be fulfilled, and let $\zeta$ be the measure appearing in $(H_2')$. The proof is divided into four steps.\smallskip

{\it Step 1.} Assume first that ${\rm cap_*}(A)<\infty$.
Since $\mathcal E^+(A)$ is strongly closed according to $(H_3)$, $\zeta^A$ and $\gamma_A$, the inner balayage of $\zeta$ to $A$ and the inner capacitary measure of $A$, respectively, are both concentrated on $A$, i.e.
\begin{equation}\label{zk}\zeta^A,\,\gamma_A\in\mathcal E^+(A)\end{equation}
(see Corollary~\ref{cor-bal} and Theorem~\ref{prop.1.2''}). We aim to show that
\begin{equation}\label{lk}
\omega:=\zeta^A+\eta_{A,f}\gamma_A,
\end{equation}
the constant $\eta_{A,f}$ being introduced by the equality in (\ref{eqalt}),
serves as the (unique) solution $\lambda_{A,f}$ to Problem~\ref{pr-main}. This will provide an alternative proof of the solvability of Problem~\ref{pr-main} (compare with Theorem~\ref{th2'} and its proof, given in Section~\ref{sec-pr}).

Indeed, applying (\ref{eq-mass}) gives
\begin{equation}\label{balbound}0\leqslant\zeta^A(X)\leqslant\zeta(X)\leqslant1,\end{equation}
the last inequality being valid by virtue of (\ref{fform'}). Thus $\eta_{A,f}\in[0,\infty)$, and combining (\ref{eqalt}), (\ref{pr1}), (\ref{zk}), and (\ref{lk}) shows that, actually, $\omega\in\breve{\mathcal E}^+(A)$. Hence, by (\ref{Isigma'}),
\[\omega\in\breve{\mathcal E}^+_f(A).\]
According to Theorem~\ref{th-ch2}, $\omega=\lambda_{A,f}$ will therefore follow once we verify the inequality
\begin{equation}\label{qk1}U_f^\omega\geqslant\int U_f^\omega\,d\omega\text{ \ n.e.\ on $A$}.\end{equation}

By use of Lemma~\ref{str}, we derive from (\ref{frr}), (\ref{eq-pr-10}), and (\ref{lk}) that the relation
\begin{equation}\label{qk2}U_f^\omega=U^\omega-U^\zeta=U^{\zeta^A-\zeta}+\eta_{A,f}U^{\gamma_A}=\eta_{A,f}\end{equation}
holds true n.e.\ on $A$, hence $\omega$-a.e.\ (see Lemma~\ref{l-negl}), and consequently
\[\int U_f^\omega\,d\omega=\eta_{A,f}\omega(X)=\eta_{A,f}.\] Substituted into (\ref{qk2}), this gives (\ref{qk1}). Thus the solution $\lambda_{A,f}$ to Problem~\ref{pr-main} does exist, and moreover $\lambda_{A,f}=\omega$ and $\eta_{A,f}=c_{A,f}$,
$c_{A,f}$ being the inner $f$-weighted equilibrium constant. This establishes (\ref{const-alt}) as well as the representation
\begin{equation}\label{pres}
\lambda_{A,f}=\zeta^A+\eta_{A,f}\gamma_A.
\end{equation}

As $(\gamma_A)^A=\gamma_A$ (Corollary~\ref{eqba}), identity (\ref{pres}) can be rewritten in the form
\begin{equation}\label{bal}
\lambda_{A,f}=(\zeta+\eta_{A,f}\gamma_A)^A,
\end{equation}
the inner balayage being additive on positive measures of finite energy.
By virtue of Theorem~\ref{th-intr} applied to $\mu:=\zeta+\eta_{A,f}\gamma_A\in\mathcal E^+$, $\lambda_{A,f}$ can therefore be characterized as the unique measure of minimum energy norm, resp.\ of minimum potential, within the class of all $\nu\in\mathcal E^+$ having the property
\[U^\nu\geqslant U^\zeta+\eta_{A,f}U^{\gamma_A}\text{ \ n.e.\ on $A$}.\]
Noting from (\ref{frr}) with the aid of Lemma~\ref{str} that this inequality is equivalent to
\[U^\nu_f\geqslant\eta_{A,f}\text{ \ n.e.\ on $A$},\]
we arrive at assertion (i), resp.\ (ii), of the theorem.

Finally, by making use of Corollary~\ref{cor-bal} we derive from (\ref{bal}) that $\lambda_{A,f}$ is the only measure in $\mathcal E^+(A)$ having the property
\[U^{\lambda_{A,f}}=U^{\zeta}+\eta_{A,f}U^{\gamma_A}\text{ \ n.e.\ on $A$},\]
or equivalently (cf.\ Lemma~\ref{str})
\[U^{\lambda_{A,f}}_f=\eta_{A,f}\text{ \ n.e.\ on $A$}.\]
This establishes assertion (iii).\smallskip

{\it Step 2.} Let now ${\rm cap}_*(A)=\infty$.
By virtue of assumption (\ref{bal1}), then necessarily $\zeta^A(X)=1$. The class $\mathcal E^+(A)$ being strongly closed according to $(H_3)$, we actually have $\zeta^A\in\breve{\mathcal E}^+(A)$ (Corollary~\ref{cor-bal}), whence (Corollary~\ref{l-aux-4})
\begin{equation}\label{zetaA}
 \zeta^A\in\breve{\mathcal E}^+_f(A).
\end{equation}
Our aim is to show that $\zeta^A$ serves as the (unique) solution to Problem~\ref{pr-main}, i.e.
\begin{equation}\label{lb}
\zeta^A=\lambda_{A,f}.
\end{equation}

By (\ref{eq-pr-11}) with $\mu:=\zeta$, $U_f^{\zeta^A}=U^{\zeta^A}-U^\zeta=0$ $\zeta^A$-a.e.,
whence
\[\int U_f^{\zeta^A}\,d\zeta^A=0,\]
which combined with (\ref{eq-pr-10}) applied to $\mu:=\zeta$ gives
\[U_f^{\zeta^A}=\int U_f^{\zeta^A}\,d\zeta^A\text{ \ n.e.\ on $A$}.\]
By Theorem~\ref{th-ch2}, this together with (\ref{zetaA}) implies (\ref{lb}) as well as $c_{A,f}=0$. Furthermore, noting from (\ref{eqalt}) that $\eta_{A,f}$ also equals $0$, we obtain (\ref{const-alt}).

We finally observe that, due to the equalities $\zeta^A=\lambda_{A,f}$ and $\eta_{A,f}=0$ thereby established, assertions (i)--(iii), providing alternative characterizations of $\lambda_{A,f}$, can be derived directly from Theorem~\ref{th-intr} with $\mu:=\zeta$.\smallskip

{\it Step 3.} The aim of this step is to verify that, under hypotheses $(H_1)$, $(H_3)$, $(H_4)$, and $(H_2')$, assumption (\ref{bal1}) is not only sufficient, but also necessary for the existence of the solution $\lambda_{A,f}$. Assume to the contrary that $\lambda_{A,f}$ exists, but (\ref{bal1}) does not hold; in view of (\ref{fform'}) and (\ref{eq-mass}), then necessarily
\begin{equation}\label{ci'}
 {\rm cap}_*(A)=\infty\text{ \ and \ }\zeta^A(X)<1.
\end{equation}

According to Lemma~\ref{la3}, for each $K\in\mathfrak C_A$ large enough ($K\geqslant K_0$), there is the solution $\lambda_{K,f}$ to Problem~\ref{pr-main} with $A:=K$; and moreover the net $(\lambda_{K,f})_{K\geqslant K_0}$ converges strongly and vaguely to the extremal measure $\xi_{A,f}$, determined by Lemma~\ref{la1}. We assert that, due to the former relation in (\ref{ci'}),
\begin{equation}\label{eqq}
 \xi_{A,f}=\zeta^A.
\end{equation}

Note that the capacity of any compact set is finite, the kernel $\kappa$ being strongly positive definite by assumption. Therefore, by (\ref{eqalt}) and (\ref{pres}) applied to $K$,
\begin{equation}\label{fol}
\lambda_{K,f}=\zeta^K+\widetilde{\eta}_{K,f}\lambda_K\text{ \ for all $K\geqslant K_0$},
\end{equation}
where $\lambda_K:=\gamma_K/{\rm cap}(K)$ is the (unique) solution to problem (\ref{W}) with $A:=K$, and
\[\widetilde{\eta}_{K,f}:=1-\zeta^K(X).\]
But the net $(\widetilde{\eta}_{K,f})_{K\geqslant K_0}\subset\mathbb R$ is bounded since, by (\ref{balbound}) with $A:=K$,
\[0\leqslant\zeta^K(X)\leqslant\zeta(X)\leqslant1\text{ \ for all $K\geqslant K_0$}.\]
Furthermore, by Theorem~\ref{th-bal-cont},
\[\zeta^K\to\zeta^A\text{ \ strongly (and vaguely) in $\mathcal E^+$ as $K\uparrow A$}.\]
Thus, if we show that
\begin{equation}\label{lconv}
\lambda_K\to0\text{ \ strongly in $\mathcal E^+$ as $K\uparrow A$},
\end{equation}
identity (\ref{eqq}) will follow from (\ref{fol}) by passing to the limit as $K\uparrow A$, and making use of the triangle inequality in the pre-Hil\-bert space $\mathcal E$.

It is seen from (\ref{153}) that the net $(\lambda_K)_{K\geqslant K_0}$ is minimizing in Problem~\ref{pr-main} with $f=0$, i.e.\ $(\lambda_K)_{K\geqslant K_0}\in\mathbb M(A)$. Applying Lemmas~\ref{la1} and \ref{la3}, we therefore conclude that there exists the unique extremal measure $\xi_A$ in Problem~\ref{pr-main} with $f=0$, and moreover $\lambda_K\to\xi_A$ strongly in $\mathcal E^+$ as $K\uparrow A$. This yields
\[\|\xi_A\|^2=\lim_{K\uparrow A}\,\|\lambda_K\|^2=\lim_{K\uparrow A}\,w(K)=0,\]
the last equality being valid due to the assumption ${\rm cap}_*(A)=\infty$. By virtue of the energy principle, $\xi_A=0$, which proves (\ref{lconv}), whence (\ref{eqq}).

Since $\lambda_{A,f}$ exists by assumption, applying Lemma~\ref{l-solv} therefore gives
\[\lambda_{A,f}=\xi_{A,f}=\zeta^A,\]
which however is impossible, for $\zeta^A(X)<1$ by (\ref{ci'}). Contradiction.\smallskip

{\it Step 4.} To complete the proof of the theorem, it remains to establish (\ref{conv3'}). Applying (\ref{eqalt}) and (\ref{const-alt}) to each $K\in\mathfrak C_A$ large enough ($K\geqslant K_0$), we get
\[c_{K,f}=\frac{1-\zeta^K(X)}{{\rm cap}(K)}.\]
In view of (\ref{conv3}), (\ref{conv3'}) will therefore follow once we show that the net $(c_{K,f})_{K\geqslant K_0}$ decreases, or equivalently, that the net $\bigl(\zeta^K(X)\bigr)_{K\in\mathfrak C_A}$ increases, cf.\ (\ref{153}). But for any $K,K'\in\mathfrak C_A$ such that $K'\geqslant K$, we have $\zeta^K=(\zeta^{K'})^K$ by Proposition~\ref{cor-rest}, hence
\[\zeta^K(X)\leqslant\zeta^{K'}(X)\]
by Proposition~\ref{cor-mass}, whence the claim.

\subsection{Proof of Corollary~\ref{cortm}}\label{sec-exc} As pointed out in Theorem~\ref{th3'}, $\lambda_{A,f}\in\Lambda_{A,f}$, the class $\Lambda_{A,f}$ being introduced by (\ref{gamma}). We thus only need to show that in the case of $\sigma$-co\-m\-p\-act $X$,
 \begin{equation}\label{ml}
 \lambda_{A,f}(X)\leqslant\mu(X),
 \end{equation}
$\mu\in\Lambda_{A,f}$ being arbitrarily given. But according to Theorem~\ref{th3'}(ii), then necessarily
\[U^{\lambda_{A,f}}\leqslant U^\mu\text{ \ everywhere on $X$},\]
and (\ref{ml}) follows by making use of
the principle of positivity of mass (Theorem~\ref{pr-pos}).

\subsection{Proof of Corollary~\ref{th5'}}\label{sec-pr5} Let hypotheses $(H_1)$, $(H_3)$, $(H_4)$, and $(H_2')$ be fulfilled, and let ${\rm cap}_*(A)=\infty$. To verify the first part of the corollary, assume moreover that $\zeta(X)<1$. In view of (\ref{eq-mass}), then
\[\zeta^A(X)\leqslant\zeta(X)<1,\] which implies, by use of Theorem~\ref{th3'}, that Problem~\ref{pr-main} indeed has no solution.

It remains to consider the case of $\zeta$ concentrated on $A$. Then the orthogonal projection of $\zeta$ onto the (strongly closed by $(H_3)$, hence strongly complete) cone $\mathcal E^+(A)$ is certainly the same $\zeta$, which implies, by virtue of Corollary~\ref{cor-bal}, that
\[\zeta^A=\zeta.\] Applying Theorem~\ref{th3'} once again, we therefore conclude that the solution $\lambda_{A,f}$ to Problem~\ref{pr-main} exists if and only if $\zeta(X)=\zeta^A(X)=1$, and in the affirmative case
\[\lambda_{A,f}=\zeta^A=\zeta,\] cf.\ the latter formula in representation (\ref{RR}). This completes the whole proof.

\subsection{Proof of Theorem~\ref{riesz}}\label{sec-pr6} Consider the $\alpha$-Riesz kernel of order $\alpha\in(0,n)$ on $\mathbb R^n$, $n\geqslant2$, a set $A\subset\mathbb R^n$ with strongly closed $\mathcal E^+(A)$, and an external field $f$ of form (\ref{fform}) such that $w_f(A)<\infty$, or equivalently with ${\rm cap}_*\bigl(\{x\in A:\ \psi(x)<\infty\}\bigr)>0$. If moreover ${\rm cap}_*(A)<\infty$, then Problem~\ref{pr-main} is solvable according to Theorem~\ref{th2'}, the $\alpha$-Riesz kernel of arbitrary order $\alpha$ being perfect. This establishes assertion (a).

If now ${\rm cap}_*(A)=\infty$, assume that $\alpha\leqslant2$; and let $f$ be of form (\ref{fform'}) --- then $w_f(A)<\infty$ is equivalent to ${\rm cap}_*(A)>0$. Furthermore, then the first and second maximum principles hold true, hence Theorem~\ref{th3'} and Corollary~\ref{th5'} can be utilized.
By a direct application of Theorem~\ref{th3'}, we thus see that $\lambda_{A,f}$ exists if and only if $\zeta^A(\mathbb R^n)=1$, and in the affirmative case $\lambda_{A,f}=\zeta^A$, cf.\ (\ref{RR}). This validates (b).

Assertion (c) is obtained as a direct application of Corollary~\ref{th5'}, or alternatively, it can be derived from (b) with the aid of arguments similar to those in Section~\ref{sec-pr5}.

The "only if" part of (d) follows directly from (c). For the "if" part, let $A$ be not inner $\alpha$-thin at infinity, and let $\zeta(\mathbb R^n)=1$. Applying \cite[Corollary~5.3]{Z-bal2} gives
\[\zeta^A(\mathbb R^n)=\zeta(\mathbb R^n)=1,\]
which implies, by making use of (b), that $\lambda_{A,f}$ does exist, and moreover $\lambda_{A,f}=\zeta^A$.

To verify (e), assume that $\overline{A}$ is $\alpha$-thin at infinity. In the Newtonian case $\alpha=2$, also assume for the sake of simplicity that $D:=(\overline{A})^c$ is connected. We aim to show that in the case
$\zeta(D)>0$, $\lambda_{A,f}$ does not exist. As seen from (b), this will follow once we prove that $\zeta^A(\mathbb R^n)<1$, which in turn is reduced to proving
\begin{equation}\label{ineq}
(\zeta|_D)^A(\mathbb R^n)<\zeta|_D(\mathbb R^n).
\end{equation}

Due to the $\alpha$-thinness of $\overline{A}$ at infinity, there exists the $\alpha$-Riesz equilibrium measure $\gamma$ of $\overline{A}$, treated in an extended sense where $I(\gamma)$ as well as $\gamma(\mathbb R^n)$ may be infinite (for details, see \cite[Section~V.1.1]{L}, cf.\ \cite[Section~5]{Z-bal} and \cite[Sections~1.3, 2.1]{Z-bal2}). Applying \cite[Theorem~8.7]{Z-bal}, we therefore get
\[(\zeta|_D)^{\overline{A}}(\mathbb R^n)<\zeta|_D(\mathbb R^n),\]
whence (\ref{ineq}), for, in consequence of Propositions~\ref{cor-rest} and \ref{cor-mass},
\[(\zeta|_D)^A(\mathbb R^n)=\bigl((\zeta|_D)^{\overline{A}}\bigr)^A(\mathbb R^n)\leqslant(\zeta|_D)^{\overline{A}}(\mathbb R^n).\]

\section{Acknowledgements} The author is deeply indebted to Bent Fuglede and Douglas P.\ Hardin
for reading and commenting on the manuscript. The present research was partially supported by the NAS of Ukraine under grant 0122U000670.

\end{document}